\documentclass[11pt,twoside]{amsart}
\usepackage{latexsym,amssymb,amsmath}
\usepackage{tikz}                         

\numberwithin{equation}{section}
\def\demo{\noindent{\it Proof. }}
\newtheorem{theorem}{Theorem}[section]
\newtheorem{lemma}[theorem]{Lemma}
\newtheorem{proposition}[theorem]{Proposition}
\newtheorem{corollary}[theorem]{Corollary}

\DeclareMathOperator{\depth}{depth}
\DeclareMathOperator{\reg}{reg}
\theoremstyle{definition}
\newtheorem{definition}[theorem]{Definition} 
 
\newtheorem{remark}[theorem]{Remark}
\newtheorem{example}[theorem]{Example}

\begin{document}


\title[Depth and regularity of monomial ideals]{Depth and regularity
of monomial ideals 
via polarization and combinatorial optimization} 
\thanks{The first and third authors were partially supported by 
SNI. The fourth author was
supported by a scholarship from CONACYT}

\author[J. Mart\'\i nez-Bernal]{Jos\'e Mart\'\i nez-Bernal}
\address{
Departamento de
Matem\'aticas\\
Centro de Investigaci\'on y de Estudios Avanzados del IPN\\
Apartado Postal
14--740 \\
07000 Mexico City, D.F.
}
\email{jmb@math.cinvestav.mx}

\author[S. Morey]{Susan Morey}
\address{Department of Mathematics\\
Texas State University\\
San Marcos, TX 78666
}
\email{morey@txstate.edu}

\author[R. H. Villarreal]{Rafael H. Villarreal}
\address{
Departamento de
Matem\'aticas\\
Centro de Investigaci\'on y de Estudios
Avanzados del
IPN\\
Apartado Postal
14--740 \\
07000 Mexico City, D.F.
}
\email{vila@math.cinvestav.mx}

\author[C. Vivares]{Carlos E. Vivares}
\address{
Departamento de
Matem\'aticas\\
Centro de Investigaci\'on y de Estudios
Avanzados del
IPN\\
Apartado Postal
14--740 \\
07000 Mexico City, D.F.
}
\email{cevivares@math.cinvestav.mx}


\begin{abstract}  
In this paper we use polarization to study the behavior of the depth
and regularity of a monomial ideal $I$, locally at a variable $x_i$, 
when we lower the degree of all the highest powers of the variable
$x_i$ occurring in the minimal generating set of $I$, and examine the
depth and regularity of powers of edge ideals of clutters using combinatorial 
optimization techniques.  If $I$ is the edge ideal of an 
unmixed clutter with the max-flow min-cut
property, we show that the powers of $I$ have non-increasing depth and
non-decreasing regularity. In particular edge ideals of unmixed bipartite
graphs have non-decreasing regularity. We are able to show that the symbolic powers of the
ideal of covers of the clique clutter of a strongly perfect graph 
have non-increasing depth. A similar result holds for the ideal of
covers of a uniform ideal clutter. 
\end{abstract}

\maketitle 

\section{Introduction}\label{introduction-section}
Let $R=K[x_1,\ldots,x_n]$ be a polynomial ring over a field $K$, let
$f$ be a monomial of $R$, and let $I\subset R$ be a monomial ideal.
The following two inequalities were shown in
\cite[Theorem~3.1]{Caviglia-et-al}:
\begin{itemize}
\item[(A)] $\depth(R/(I\colon f))\geq\depth(R/I)$,
\item[(B)] $\reg(R/I)\geq\reg(R/(I\colon f))$,  
\end{itemize}
where $\depth(R/I)$ and $\reg(R/I)$ are the depth and regularity of
the quotient ring $R/I$ and $(I\colon f)=\{g\in R \vert gf \in I \}$ is referred to as
a colon ideal. 
If $I$ and $f$ are
squarefree, we show that (A) and (B) are equivalent using a duality
theorem of Terai \cite{terai} (Theorem~\ref{terai-duality-theorem})
and some duality formulas for edge ideals of clutters
(Lemma~\ref{duality-formula}), 
that is, (A) and (B) are dual statements in the squarefree case
(Proposition~\ref{square-free-case}). 

We introduce a formula expressing  $\depth(R/(I,f))-\depth(R/I)$,
$\reg(R/I)$ and
$\reg(R/(f,I))$ in terms of the depth and regularity of polarizations
(Proposition~\ref{pepe-morey-vila-vivares}). Then, as an application, 
we give an alternate proof of $(A)$ and $(B)$, and show some other
known inequalities about depth and regularity 
(Corollary~\ref{Caviglia-et-al-lemma-alternate-proof}). If
${\rm in}_\prec(I+f)=I+{\rm in}_\prec(f)$ for some monomial
order $\prec$ and some homogeneous polynomial $f$, we show that (A) and (B)
hold (Corollary~\ref{feb20-18}).  

The aim of this paper is to use these results to study the behavior of the depth and
regularity of $R/I$, locally at a variable $x_i$, when we lower the degree of all the
highest powers of the variable $x_i$ occurring in the minimal
generating set of $I$ and, furthermore, to examine the
depth and regularity of powers and symbolic powers of edge ideals of
clutters and graphs, and their ideals of covers, using combinatorial 
optimization techniques. 

Fix a variable $x_i$ that occurs in the minimal generating
set $G(I)$ of $I$. Let $q$ be the maximum of the degrees in $x_i$ of the
monomials of $G(I)$, let $\mathcal{B}_i$ be the set of all monomials 
of $G(I)$ of degree $q$ in $x_i$, let 
$p$ be the maximum of the degrees in $x_i$ of the
monomials of $\mathcal{A}_i=G(I)\setminus\mathcal{B}_i$, and consider
the ideal $L=(\{x^a/x_i\vert\,x^a\in\mathcal{B}_i\}\cup\mathcal{A}_i\})$. 

One of our main results shows that the depth is locally non-decreasing at
each variable $x_i$ when lowering the top degree. Note that if $p=0$,
that is, if all generators of $I$ that are divisible by $x_i$ have
degree $q$ in $x_i$, then $L=(I\colon x_i)$. Thus when $p=0$ we have
from $(A)$ that $\depth(R/L) = \depth(R/(I\colon x_i)) \geq \depth(R/I)$.
This theorem allow control over the depth when the degrees in $x_i$
of the generators varies.

\noindent {\bf Theorem~\ref{morey-vila-oaxaca-2017}}{\it\  
{\rm (a)} If $p\geq 1$ and $q-p\geq 2$, then
$\depth(R/I)=\depth(R/L)$.  

{\rm(b)} If $p\geq 0$ and $q-p=1$, then $\depth(R/L)\geq\depth(R/I)$. 

{\rm (c)} If $p=0$ and $q\geq 2$, then 
$\depth(R/I)=
\depth(R/(\{x^a/x_i^{q-1}\vert\,x^a\in\mathcal{B}_i\}\cup\mathcal{A}_i\}))$.}

There are similar results for regularity
(Theorem~\ref{morey-vila-oaxaca-2017-reg}). As a consequence one
recovers a result of Herzog, Takayama and Terai
\cite{herzog-takayama-terai} 
showing that 
$\depth(R/{\rm rad}(I))\geq\depth(R/I)$ and a result of Ravi
\cite{ravi} showing that $\reg(R/{\rm 
rad}(I))\leq\reg(R/I)$ (Corollaries~\ref{herzog-takayama-terai-theo}
and \ref{feb21-18}). The result can also be used to show that the
Cohen--Macaulay property of a vertex-weighted digraph is dependent
only on knowing which vertices have weight greater than one and not
on the actual weights used (Corollary~\ref{oriented-graphs}).    

There are some classes of monomial ideals whose powers have non-increasing depth and
non-decreasing regularity  
\cite{Caviglia-et-al,constantinescu-etal,Nguyen-Thu-Hang,
Hang-Trung,very-well-covered-non-inc,very-well-covered-non-inc-proc}. A natural way to show these 
properties for a monomial ideal $I$ is to prove the existence of a
monomial $f$ such that $(I^{k+1}\colon f)=I^k$ for $k\geq 1$. This was
exploited in \cite{Caviglia-et-al,edge-ideals} and in \cite[Corollary~3.11]{Ha-Morey} in
 connection to normally torsion-free ideals. 

Since any squarefree monomial
ideal is the edge ideal $I(\mathcal{C})$ of a clutter $\mathcal{C}$, 
we will study the depth and regularity of powers and symbolic powers of edge ideals of
clutters and graphs--and their ideals of covers---that have nice combinatorial optimization
properties (e.g., max-flow min-cut, ideal, uniform, and unmixed clutters, strongly perfect and very
well-covered graphs). The $k$-th symbolic power of an ideal $I$ is denoted by
$I^{(k)}$ (Definition~\ref{symbolic-power-def}). The {\it ideal of
covers\/} of a clutter $\mathcal{C}$, denoted 
$I(\mathcal{C})^\vee$, is the edge ideal of $\mathcal{C}^\vee$, the
clutter of minimal vertex covers of $\mathcal{C}$.

If $I(\mathcal{C})$ is the edge ideal of a clutter $\mathcal{C}$ which has a good
leaf, then the powers of $I(\mathcal{C})$ have non-increasing depth and
non-decreasing regularity \cite[Theorem~5.1]{Caviglia-et-al}. 
In particular edge ideals of forests or simplicial trees 
have these properties. Our next
result gives a wide family of ideals with these properties. 

\noindent {\bf Theorem~\ref{winter2017}}{\it\ If $I=I(\mathcal{C})$ is 
the edge ideal of an unmixed clutter $\mathcal{C}$ with the max-flow min-cut property, then 
\begin{itemize}
\item[(a)] $\depth(R/I^k)\geq\depth(R/I^{k+1})$ for $k\geq 1$, and 
\item[(b)] $\reg(R/I^k)\leq\reg(R/I^{k+1})$ for $k\geq 1$.
\end{itemize}}

Let $G$ be a graph with vertex set $V(G)=\{x_1,\ldots,x_n\}$ and edge
set $E(G)$. A result of T. N. Trung \cite{Tran-Nam-Trung} shows that
for $k\gg 0$ one has 
$$
\depth(R/I(G)^k)=|{\rm isol}(G)|+c_0(G),
$$
where ${\rm isol}(G)$ is the set of isolated vertices of $G$ and
$c_0(G)$ is the number of non-trivial bipartite components of $G$. We
complement this fact by observing that 
$\dim(R)-\ell(I(G))$ is equal to $|{\rm isol}(G)|+c_0(G)$, where $\ell(I(G))$ is the
analytic spread of $I(G)$, and by showing the 
inequality 
$$\depth(R/(I(G)^k\colon x_i^k))\leq\depth(R/(I(G\setminus
N_G(x_i))^k,N_G(x_i)))$$
for $k\geq 1$ and $i=1,\ldots,n$
(Proposition~\ref{trung-limit-powers}), where $N_G(x_i)$ is the
neighbor set of $x_i$. For $k=1$ this inequality
follows from the fact that $(I(G)\colon x_i)$ is equal to $(I(G\setminus
N_G(x_i)),N_G(x_i))$ \cite[p. 293]{monalg-rev} and using the
inequality $\depth(R/(I(G)\colon x_i))\geq \depth(R/I(G))$. The
general case follows by successively applying 
Theorem~\ref{morey-vila-oaxaca-2017} locally at each variable. 

It is an open problem whether or not the powers of the edge ideal of
a graph have  
non-increasing depth. To the best of our knowledge this is open even
for bipartite graphs. Our next application extends the fact that
the powers of $I(G)^\vee$, the ideal of covers of $G$, have non-increasing depth if $G$ is bipartite
\cite{constantinescu-etal,Nguyen-Thu-Hang,Hang-Trung}.

\noindent {\bf Corollary~\ref{vila-hang}}{\it\ Let $G$ be a bipartite graph. The following hold. 
\begin{itemize}
\item[\rm(a)] \cite[Corollary~5.3]{kimura-terai-yassemi} If $G$ is unmixed, then $I(G)$ 
has non-increasing depth.
\item[\rm(b)] $($\cite[Theorem~3.2]{constantinescu-etal}, \cite{Nguyen-Thu-Hang},
\cite[Corollary~2.4]{Hang-Trung}$)$ $I(G)^\vee$ has non-increasing depth.
\item[\rm(c)] $I(G)^\vee$ has non-decreasing regularity. 
\end{itemize}}
\quad An interesting example due to Kaiser,
Stehl\'\i k, and  \v{S}krekovski \cite{persistence-ce} shows that
the powers of the ideal of covers of a graph does not always have non-increasing
depth (Example~\ref{kaiser}), that is, part (b) of 
Corollary~\ref{vila-hang} fails for non-bipartite graphs. A nice
result of L. T. Hoa, K. Kimura, N. Terai and T. N. Trung
\cite[Theorem~3.2]{Hoa-etal} shows that
the symbolic powers of the ideal of covers of a graph have
non-increasing depth. A similar result holds for the ideal of covers 
of a uniform ideal clutter (Corollary~\ref{winter2017-dual}).

If $G$ is a very well-covered graph, 
then the depths of symbolic powers of $I(G)^\vee$ form a 
non-increasing sequence \cite{very-well-covered-non-inc} (cf.
\cite[Theorem~3.2]{Hoa-etal}) and also the
depths of symbolic powers of
$I(G)$ form a non-increasing sequence
\cite[Theorem~5.2]{kimura-terai-yassemi}.   
In this case we show that the symbolic powers of
$I(G)$ have non-decreasing regularity
(Proposition~\ref{very-well-covered}). 

We will give another family of squarefree monomial ideals whose symbolic powers have
non-increasing depth and non-decreasing regularity.  A {\it clique\/} of a graph $G$
is a set of vertices inducing a complete subgraph. The {\it clique
clutter\/} of $G$, denoted by ${\rm cl}(G)$, is the clutter on $V(G)$ whose edges are the 
maximal cliques of $G$. 

\noindent {\bf Proposition~\ref{strongly-perfect-symbolic}}{\it\ Let $G$ be a strongly perfect graph
and let ${\rm cl}(G)$ be its clique clutter. If $J$ is the ideal of
covers of ${\rm cl}(G)$, then
\begin{itemize}
\item[(a)] $\depth(R/J^{(k)})\geq\depth(R/J^{(k+1)})\ \mbox{ for }\
k\geq 1$, and 
\item[(b)] $\reg(R/J^{(k)})\leq\depth(R/J^{(k+1)})\ \mbox{ for }\
k\geq 1$.
\end{itemize}
}
\quad Bipartite graphs, chordal graphs, comparability graphs, and Meyniel
graphs are strongly perfect (see \cite{ravindra-strongly-perfect} and
the references therein). Thus this result generalizes 
Corollary~\ref{vila-hang}(b) because if $G$ is a bipartite graph, then 
${\rm cl}(G)=G$ and $I(G^\vee)^{(k)}=I(G^\vee)^k$ for $k\geq 1$
\cite{reesclu}. 

For edge ideals of clutters the Cohen--Macaulay property of its 
$k$-th ordinary or symbolic power is well understood if $k\geq 3$. By a
result of N. Terai and N. V. Trung \cite{terai-trung}, if
$I(\mathcal{C})$ is the edge ideal of a clutter $\mathcal{C}$, then
$I(\mathcal{C})^k$ 
(resp. $I(\mathcal{C})^{(k)}$) is Cohen--Macaulay for some $k\geq 3$ if and
only if $I(\mathcal{C})$ is a complete intersection (resp. the 
independence complex $\Delta_\mathcal{C}$ of $\mathcal{C}$ is a
matroid).  

The case when $G$ is a graph and $k=2$ is treated in
\cite{crupi-et-al,Hoang-etal,hoang-gorenstein-second-jaco,trung-tuan}. The Cohen--Macaulay
property of the square of an  
edge ideal can be expressed in terms of its connected components
\cite{Ha-sum-powers-of-sums,rinaldo-terai-yoshida}
(Corollary~\ref{cm-square-components}).  
Edge ideals of graphs whose square is Cohen--Macaulay have a rich 
combinatorial structure and have been 
classified combinatorially by D. T. Hoang, N. C. Minh and T. N. Trung 
\cite{Hoang-etal,hoang-gorenstein-second-jaco}. The Cohen-Macaulay property of $I(G)^2$ 
is also studied in \cite{trung-tuan} in terms
of simplicial complexes.

As an application we recover the following fact. 

\noindent {\bf Corollary~\ref{square-cm}}
{\rm\ (\cite[Theorem~2.7]{crupi-et-al}, 
\cite[Proposition~4.2]{Hoang-etal})}
{\it\ Let $G$ be a bipartite graph
without isolated vertices. Then $I(G)^2$ is Cohen-Macaulay if and
only if $G$ is a disjoint union of edges.
}

For all unexplained
terminology and additional information  we refer to 
\cite{Eisen,Mats} (for commutative algebra), \cite{cornu-book,Schr,Schr2} (for
combinatorial optimization),   
\cite{Har} (for graph theory), 
and
\cite{francisco-ha-mermin,graphs-rings,Herzog-Hibi-book,chapter-vantuyl,monalg-rev} 
(for the theory of powers of edge ideals of clutters and monomial ideals). 

\section{Depth and regularity of monomial ideals via polarization}\label{depth-monomial}

Let $R=K[x_1,\ldots,x_n]$ be a polynomial ring over
a field $K$ and let $I$ be a monomial ideal. The unique minimal set
of generators of $I$  
consisting of monomials is denoted by $G(I)$. The goal of this section is to use polarization to control the depth and regularity of $R/I$ when the powers of a variable appearing in $G(I)$ are reduced. 
To do so, we first recall some known results, then show a series of equivalent conditions that will allow us to study the behavior of the depth and the regularity of $R/I$. 

In \cite[Lemma~5.1]{Dao-Huneke-Schweig} it was shown that
$\depth(R/(I\colon x_i)) \geq \depth(R/I)$ for all $i$. By noting that a
generating set for $(I \colon x_i)$ can be found from $G(I)$ by
reducing all powers of $x_i$ by one, this can be viewed as the first
step in reaching the goal.  The result was recently generalized in   
\cite[Theorem~3.1]{Caviglia-et-al} to any monomial ideal. We provide
an alternate proof using polarization. We begin by treating the
squarefree case using Stanley-Reisner complexes.   

Recall that if $\Delta$ is a simplicial complex with vertices
$x_1,\ldots,x_n$, the
{\it Stanley-Reisner ideal} of $\Delta$, denoted by $I_{\Delta}$, is
the ideal of $R$ whose squarefree monomial generators correspond to non-faces 
of $\Delta$. That is, $$I_{\Delta} = (x_{i_1}\cdots x_{i_t} \vert
\{x_{i_1}, \ldots , x_{i_t}\} \not\in \Delta).$$ 
\quad The following result
shows how the structure of the simplicial complex can be used to find
the depth of the associated ideal.   

\begin{theorem}{\rm \cite{dsmith}}\label{aug21-01}
Let $\Delta$ be a simplicial complex with vertex set
$V=\{x_1,\ldots,x_n\}$, let $I_\Delta$ be its
Stanley--Reisner ideal, and $K[\Delta]=R/I_\Delta$. Then 
$$
{\rm depth}(R/I_\Delta)=1+\max\{i\, \vert\, K[\Delta^i]
\mbox{ is Cohen--Macaulay}\},
$$
where $\Delta^i=\{F\in\Delta\, \vert\, 
\dim(F)\leq i\}$ is the $i$-skeleton and $-1\leq
i\leq\dim(\Delta)$.
\end{theorem}

The {\it star} of a face 
$\sigma$ 
in a simplicial complex $\Delta$, denoted  ${\rm
star}_\Delta(\sigma)$, is defined to be
the subcomplex of $\Delta$ generated by all facets of $\Delta$ that 
contain $\sigma$.

\begin{lemma}{\cite[Theorem~3.1]{Caviglia-et-al}}
\label{Caviglia-et-al-square-free} 
Let $I\subset R$ be a squarefree monomial ideal
and let $f$ be a squarefree monomial. Then 
$\depth(R/(I\colon f))\geq\depth(R/I)$.
\end{lemma}

\begin{proof} Let $\sigma={\rm supp}(f)$ be the set of all variables
that occur in $f$. We may assume that $f$ is a zero divisor of $R/I$
because otherwise $(I\colon f)=I$ and there is nothing to prove.  We
may also assume that $f$ is not in all minimal primes of $I$ because
in this case $(I\colon f)=R$ and $\depth(0)=\infty$. Let $\Delta$ 
and $\Delta'$ be the Stanley--Reisner complexes of $I$ and 
$(I\colon f)$, respectively. Setting $d=\dim(\Delta)$,
$d'=\dim(\Delta')$, one has $d'\leq d$. Assume that $\Delta^i$ is
Cohen--Macaulay for some $i\leq d$. We claim that $i\leq d'$. If
$i>d'$, take a facet $F$ of $\Delta'$ of dimension $d'$, that is, $F$
is a facet of $\Delta$ of dimension $d'$ containing $\sigma$. As $F$ is 
a face of $\Delta^i$ and this complex is pure, we get that $F$ is
properly contained in a face of $\Delta$ of dimension $i$, a
contradiction. Hence $i\leq d'$. The simplicial complex $\Delta'$ is
equal to ${\rm star}_{\Delta}(\sigma)$. Therefore, from the equalities
$$
(\Delta')^i=({\rm star}_{\Delta}(\sigma))^i={\rm
star}_{\Delta^i}(\sigma),
$$
and using that the star of a face of a Cohen--Macaulay complex is
again Cohen--Macaulay 
\cite[p.~224]{monalg-rev}, we get that $(\Delta')^i$ is
Cohen--Macaulay. Hence, by Theorem~\ref{aug21-01}, it follows that 
the depth of $R/(I\colon f)$ is greater than or equal to $\depth(R/I)$.
\end{proof}

A common technique in commutative algebra is to start with a short exact sequence of the form
$$0 \longrightarrow R/(I\colon f)[-k]\stackrel{f}{\longrightarrow}R/I
\longrightarrow R/(I,f) \longrightarrow 0,$$ 
where $I\subset R$ is a graded ideal and $f$ is a homogeneous
polynomial of degree $k$, and use information about two of the terms to glean desired information about the third. Both depth and regularity are known to behave well relative to short exact sequences. There are several versions of the depth lemma that appear in the literature. The following lemmas provide the information relating the depths and regularity of the terms of a short exact sequence in a format that will be particularly useful in the remainder of this paper.

\begin{lemma}\label{depth-lemma-duality}
Let $0\rightarrow N\rightarrow M\rightarrow L\rightarrow 0$ be a 
short exact sequence of modules over a local ring $R$. The following 
conditions are equivalent.
\begin{enumerate}
\item[\rm(a)] $\depth(N)\geq \depth(M)$.
\item[\rm(b)] $\depth(M)=\depth(N)$ or $\depth(M)=\depth(L)$.
\item[\rm(c)] $\depth(L)\geq\depth(M)-1$.
\end{enumerate}
\end{lemma}
\begin{proof} It follows from the depth lemma
\cite[Lemma~2.3.9]{monalg-rev}. 
\end{proof}
There is a similar statement for the regularity.
\begin{lemma}\label{reg-lemma-special}
Let $0\rightarrow N\rightarrow M\rightarrow L\rightarrow 0$ be a  
short exact sequence of graded finitely generated $R$-modules. 
The following conditions are equivalent.
\begin{itemize}
\item[\rm(a)] ${\rm reg}(M)\geq{\rm reg}(N)-1$.
\item[\rm(b)] ${\rm reg}(M)={\rm reg}(N)$ or ${\rm reg}(M)={\rm reg}(L)$.
\item[\rm(c)] ${\rm reg}(M)\geq{\rm reg}(L)$.
\end{itemize}
\end{lemma}
\begin{proof} It follows from \cite[Corollary~20.19]{Eisen}.
\end{proof}
\begin{lemma}\label{reg-lemma-special-k}
Let $0\rightarrow N\rightarrow M\rightarrow L\rightarrow 0$ be an 
exact sequence of graded finitely generated $R$-modules with
homomorphisms of degree $0$ and $k\geq 1$ an integer.  
The following are equivalent.
\begin{itemize}
\item[\rm(a)] ${\rm reg}(N)\leq{\rm reg}(M)+k$.
\item[\rm(b)] ${\rm reg}(L)\leq{\rm reg}(M)+k-1$.
\end{itemize}
\end{lemma}

\demo (a) $\Rightarrow$ (b): We may assume
$\reg(M)\leq\reg(L)-1$, otherwise there is nothing to prove. Hence, by
\cite[Corollary~20.19]{Eisen}, we get
$$
\reg(L)\leq\max(\reg(N)-1,\reg(M))\leq\reg(M)+k-1.
$$
\quad (b) $\Rightarrow$ (a): As $\reg(L)+1\leq\reg(M)+k$, by
\cite[Corollary~20.19]{Eisen}, we get
$$
\reg(N)\leq\max(\reg(M),\reg(L)+1)\leq\reg(M)+k.\eqno\Box 
$$

Let $\mathcal{C}$ be a {\it clutter\/} with vertex 
set $X=\{x_1,\ldots,x_n\}$, that is, $\mathcal{C}$ consists of a 
family of subsets of $X$, called edges, none of which is included in
another. The sets of
vertices and edges of $\mathcal C$ are denoted by $V(\mathcal{C})$ and
$E(\mathcal{C})$, respectively. If $V\subset X$, 
the clutter obtained from $\mathcal{C}$ 
by deleting all edges of $\mathcal{C}$ that intersect $V$ will be
denoted by $\mathcal{C}\setminus V$. The {\it edge ideal\/} 
of $\mathcal{C}$, denoted $I(\mathcal{C})$,
is the ideal of $R$ generated by all  
squarefree monomials $x_e=\prod_{x_i\in e}x_i$ such that $e\in E(\mathcal{C})$.
The {\it ideal of covers\/} $I(\mathcal{C})^\vee$ of $\mathcal{C}$ is
the edge ideal of $\mathcal{C}^\vee$, the
clutter of minimal vertex covers of $\mathcal{C}$ \cite[p.~221]{monalg-rev}. The ideal
$I(\mathcal{C})^\vee$ is also called the {\it Alexander dual\/} of
$I(\mathcal{C})$ or simply the {\it cover ideal} of $\mathcal{C}$. 

\begin{lemma}\label{duality-formula} Let $I(\mathcal{C})\subset R$ be the edge
ideal of a clutter $\mathcal{C}$ and let $f=x_{i_1}\cdots x_{i_k}$ be a
squarefree monomial of $R$. The following hold.
\begin{enumerate}
\item[(i)] $(I(\mathcal{C})^\vee\colon f)^\vee=
I(\mathcal{C}\setminus\{x_{i_1},\ldots,x_{i_k}\})$.
\item[(ii)] $(I(\mathcal{C})\colon f)^\vee=
I(\mathcal{C}^\vee\setminus\{x_{i_1},\ldots,x_{i_k}\})$.
\item[(iii)] If $x_i$ is a variable, then
$(I(\mathcal{C}),x_i)^\vee=x_iI(\mathcal{C}\setminus\{x_i\})^\vee$.
\end{enumerate}
\end{lemma}

\begin{proof} 
(i): Let $E(\mathcal{C})$ be the set of edges of
$\mathcal{C}$. We set $V=\{x_{i_1},\ldots,x_{i_k}\}$ and
$I=I(\mathcal{C})$. 
Then
\begin{equation*}
(I^\vee\colon f)^\vee=\left(\bigcap_{e\in E(\mathcal{C})}(e)\colon 
f\right)^\vee=\left(\bigcap_{e\in E(\mathcal{C}\setminus
V)}(e)\right)^\vee=(I(\mathcal{C}\setminus V)^\vee)^\vee
=I(\mathcal{C}\setminus V).
\end{equation*}
\quad (ii): Notice the equalities 
$I(\mathcal{C}^\vee)^\vee=(I(\mathcal{C})^\vee)^\vee=I(\mathcal{C})$.
Thus this part follows from (i) by replacing $\mathcal{C}$ with
$\mathcal{C}^\vee$. 

(iii): Setting $L=(I(\mathcal{C}),x_i)$ and
$J=I(\mathcal{C}\setminus\{x_i\})$, it follows 
readily that
$$ 
L=(I(\mathcal{C}\setminus\{x_i\}),x_i)=(J,x_i)=\bigcap_{\mathfrak{p}\in{\rm
Ass}(R/J)}(x_i,\mathfrak{p}).
$$
\quad Hence, by duality \cite[Theorem~6.3.39]{monalg-rev}, 
one has $(I(\mathcal{C}),x_i)^\vee=x_iI(\mathcal{C}\setminus\{x_i\})^\vee$.
\end{proof}

Our interest in the duality results above is partially motivated by
the following result relating regularity and projective dimension,
and thus depth, when passing to the dual.  

\begin{theorem}{\rm(Terai \cite{terai})}
\label{terai-duality-theorem} If $I\subset R$ is a 
squarefree monomial ideal, then 
$$
{\rm reg}(I)=1+{\rm reg}(R/I)={\rm pd}(R/I^\vee).
$$
\end{theorem}

In \cite[Theorem~3.1]{Caviglia-et-al} it is shown that conditions
(ii) and (iv) of the next result hold (cf. \cite[Lemmas~5.1 and 2.10]{Dao-Huneke-Schweig}). 
For squarefree monomial ideals---using the above duality
theorem of Terai \cite{terai}---we show that these conditions are in
fact equivalent (cf. Remark~\ref{sep17-17}). Roughly speaking the 
inequalities of (ii) and (iv) are dual of each other via the duality
theorem of Terai. 

\begin{proposition}\label{square-free-case} Let $I\subset R$ be a
squarefree monomial ideal and let $f=x_{i_1}\cdots x_{i_k}$ be a
squarefree monomial of $R$ of degree $k$. Then any of the following
equivalent conditions hold.
\begin{enumerate}
\item[(i)] $\depth(R/(f,I))\geq\depth(R/I)-1$.
\item[(ii)] \cite[Theorem~3.1]{Caviglia-et-al} $\depth(R/I)\leq\depth(R/(I\colon f))$.
\item[(iii)] $\depth(R/(x_{i_1},\ldots,x_{i_k},I))\geq \depth(R/I)-k$.
\item[(iv)] \cite[Theorem~3.1]{Caviglia-et-al} $\reg(R/I)\geq\reg(R/(I\colon f))$.
\item[(v)] $\reg(R/(f,I))\leq\reg(R/I)+k-1$.
\end{enumerate}
\end{proposition}

\begin{proof} 
By Lemma~\ref{Caviglia-et-al-square-free}, condition (ii)
holds for any squarefree monomial ideal $I$ and for any squarefree
monomial $f$. Thus it suffices to show that (i) and (ii) are equivalent and that (i) and (iii)--(v) are equivalent
conditions. Since $I$ is squarefree, 
there is a clutter $\mathcal{C}$ such that $I=I(\mathcal{C})$. 

(i) $\Leftrightarrow$ (ii):  This follows from applying
Lemma~\ref{depth-lemma-duality} to the short exact sequence
\begin{equation}\label{sep15-17-1}
0\longrightarrow R/(I\colon
f)[-k]\stackrel{f}{\longrightarrow}R/I\longrightarrow
R/(I,f)\longrightarrow 0.
\end{equation}

(i) $\Rightarrow$ (iii): This follows directly by induction on $k$.

(iii) $\Rightarrow$ (iv): As (iii) holds for squarefree monomials, 
applying (iii) to $I(\mathcal{C}^\vee)$, we get  
$$ 
k+\depth(R/(x_{i_1},\ldots,x_{i_k},I(\mathcal{C}^\vee)))\geq
\depth(R/I(\mathcal{C}^\vee)).
$$
\quad Therefore, setting $V=\{x_{i_1},\ldots,x_{i_k}\}$ and
$X=\{x_1,\ldots,x_n\}$, we get 
\begin{eqnarray*}
\depth(R/I(\mathcal{C}^\vee\setminus V))&=&k+\depth(K[X\setminus
V]/I(\mathcal{C}^\vee\setminus V))\\
&=&k+\depth(R/(V,I(\mathcal{C}^\vee)))
\geq\depth(R/I(\mathcal{C}^\vee)),
\end{eqnarray*}
that is, $\depth(R/I(\mathcal{C}^\vee\setminus V))\geq
\depth(R/I(\mathcal{C}^\vee))$, where
$I(\mathcal{C}^\vee)=I(\mathcal{C})^\vee$. Hence, applying 
the Auslander--Buchsbaum formula \cite[Theorem~3.5.13]{monalg-rev} to
both sides of this inequality and then using Terai's formula of
Theorem~\ref{terai-duality-theorem}, we get 
$$
\reg(R/I(\mathcal{C}))\geq\reg(R/I(\mathcal{C}^\vee\setminus V)^\vee).
$$
\quad By Lemma~\ref{duality-formula}(ii) one has 
$I(\mathcal{C}^\vee\setminus V)=(I(\mathcal{C})\colon f)^\vee$. Thus,
by duality, $I(\mathcal{C}^\vee\setminus V)^\vee
=(I(\mathcal{C})\colon f)$, and the required inequality follows.  

(iv) $\Rightarrow$ (iii): As (iv) holds for squarefree monomials, 
applying (iv) to $I(\mathcal{C}^\vee)$, we get  
$$
\reg(R/I(\mathcal{C}^\vee))\geq\reg(R/(I(\mathcal{C}^\vee)\colon f)).
$$
\quad Therefore, applying Terai's formula of
Theorem~\ref{terai-duality-theorem} and
Lemma~\ref{duality-formula}(i), we get 
$$
{\rm pd}_R(R/I(\mathcal{C}))\geq {\rm pd}_R(R/I(\mathcal{C}\setminus
V)).
$$
\quad Hence, applying 
the Auslander--Buchsbaum formula \cite[Theorem~3.5.13]{monalg-rev} to
both sides of this inequality and using depth properties, we obtain  
\begin{eqnarray*}
k+\depth(R/(V,I(\mathcal{C})))&=&k+\depth(K[X\setminus V]/I(\mathcal{C}\setminus
V))\\
&=&\depth(R/I(\mathcal{C}\setminus V))\geq\depth(R/I(\mathcal{C})).
\end{eqnarray*}

(iv) $\Leftrightarrow$ (v):  Since $\reg((R/(I\colon
f))[-k])=k+\reg(R/(I\colon f))$, the equivalence between (iv) and (v)
follows applying Lemma~\ref{reg-lemma-special-k} to the exact sequence of 
Eq.~(\ref{sep15-17-1}). 
\end{proof}

In {\cite[Corollary~3.3]{Caviglia-et-al}} it is shown that condition
(vii) below holds (cf. \cite[Lemma~2.10]{Dao-Huneke-Schweig}). 

\begin{remark}\label{sep17-17} (A) Conditions {\rm (i)--(v)} are
equivalent to 

{\rm (vi)} $\depth(R/I)=\depth(R/(I\colon f))$ or $\depth(R/I)=
\depth(R/(f,I))$.

\quad (B) For $k=\deg(f)=1$ conditions {\rm (i)--(vi)} are equivalent to:

{\rm (vii)} $\reg(R/I)=\reg(R/(I\colon f))+1$ or
$\reg(R/I)=\reg(R/(f,I))$. 

This follows applying Lemmas~\ref{depth-lemma-duality} and
\ref{reg-lemma-special} to the exact sequence given in Eq.~(\ref{sep15-17-1}). 
\end{remark}

\paragraph{\bf Depth and regularity via polarization} 
In what follows we will use the polarization technique
due to Fr\"oberg that we briefly recall now (see \cite[p.~203]{monalg-rev} and the references
therein). Note that alternate labelings of polarizations and partial
polarizations exist in the literature (see, for example,
\cite{faridi-lisbon,Herzog-Hibi-book,peeva}); 
however, the notation used here will prove beneficial
in Section~\ref{depths}.   

Let $J\subset R$ be a monomial ideal minimally generated by 
$G(J)=\{g_1,\ldots,g_s\}$. We set 
$\gamma_i$ equal to $\max\{\deg_{x_i}(g)\vert\, g\in G(J)\}$. 
To polarize $J$ we use the set of new variables
$$
X_J=\cup_{i=1}^n\{x_{i,2},\ldots,
x_{i,\gamma_i}\},
$$
where $\{x_{i,2},\ldots,x_{i,\gamma_i}\}$ is empty if
$\gamma_i=0$ or $\gamma_i=1$. It
is convenient to identify the variable $x_i$ with $x_{i,1}$ for all
$i$. Recall that a power $x_i^{c_i}$ of a variable $x_i$, $1\leq
c_i\leq\gamma_i$, polarizes 
to $(x_i^{c_i})^{\rm pol}=x_i$ if $\gamma_i=1$, to
$(x_i^{c_i})^{\rm pol}=x_{i,2}\cdots x_{i,c_i+1}$ if $c_i<\gamma_i$, and
to $(x_i^{c_i})^{\rm pol}=x_{i,2}\cdots x_{i,\gamma_i}x_i$ if $c_i=\gamma_i$.
This induces a polarization $g_i^{\rm pol}$ of $g_i$ for $i=1,\ldots,s$.  
The {\it full polarization\/} $J^{\rm pol}$ of $J$ is the ideal
of $R[X_J]$ generated by $g_1^{\rm pol},\ldots,g_s^{\rm pol}$. The next lemma
is well known.

\begin{lemma}\label{sep10-17} Let $J$ be a monomial ideal of $R$. Then
\begin{enumerate}
\item[(a)] {\rm (Fr\"oberg \cite{Fro1})}
$\depth(R[X_J]/J^{\rm pol})=|X_J|+\depth(R/J)=\depth(R[X_J]/J)$. 
\item[(b)] ${\rm pd}(R/J)={\rm pd}(R[X_J]/J^{\rm pol})$.
\item[(c)] ${\rm pd}(R/J)={\rm reg}(R[X_J]/(J^{\rm pol})^\vee)+1$.
\item[(d)] {\rm \cite[Corollary 1.6.3]{Herzog-Hibi-book}}
$\reg(R/J)=\reg(R[X_J]/J^{\rm pol})$.
\end{enumerate}
\end{lemma}

\begin{proof} Part (b) follows applying the Auslander--Buchsbaum
formula \cite[Theorem~3.5.13]{monalg-rev} to part (a). Part (c)
follows from Theorem~\ref{terai-duality-theorem} and part (b).
\end{proof}

Let $I\subset R$ be a monomial ideal and let $f$ be a monomial. Using
polarization, one can extend Proposition~\ref{square-free-case} 
and Remark~\ref{sep17-17} to general monomial ideals.
The following result will be needed when relating the depth and
the regularity of a monomial 
ring $R/I$ with those of the ring $R[X_L]/I^{\rm pol}$, where $L$ is the
ideal $(f,I)$
and $I^{\rm pol}$ is the polarization of $I$ with respect to 
$R[X_L]$ (cf. Lemma~\ref{sep10-17}).  

\begin{proposition}\label{pepe-morey-vila-vivares} 
Let $I\subset R$ be a monomial ideal and
let $f$ be a monomial.  
If $L=(f,I)$ and $X_L$ is the set of new 
variables that are needed to polarize $L$, then
\begin{enumerate}
\item[(i)] $\depth(R[X_L]/L^{\rm pol})
-\depth(R[X_L]/(f_1^{\rm pol},\ldots,f_r^{\rm pol}))=\depth(R/L)-\depth(R/I)$,
\item[(ii)] $\reg(R[X_L]/L^{\rm pol})=\reg(R/L)$ and 
$\reg(R[X_L]/(f_1^{\rm pol},\ldots,f_r^{\rm pol}))=\reg(R/I)$,
\end{enumerate}
where $G(I)=\{f_1,\ldots,f_r\}$, and $f_i^{\rm pol}$ is the polarization of
$f_i$ in $R[X_L]$. 
\end{proposition}

\begin{proof} (i): We may assume $f$ is not in $I$, otherwise
there is nothing to prove. Let $L^{\rm pol}\subset R[X_L]$
be the full polarization of $L$. For use below we set 
$\delta_i=\max\{\deg_i(g)\vert\, g\in G(I)\}$ and 
$f=x_1^{a_1}\cdots x_n^{a_n}$. The set of variables of $R$ is denoted
by $X=\{x_1,\ldots,x_n\}$.

Subcase (i.a): $a_i>\delta_i$ for some $i$. Then
$G(L)=\{f,f_1,\ldots,f_r\}$. For simplicity of notation we 
assume there is an integer $k$ such that
$a_1>\delta_1,\ldots,a_k>\delta_k$  and $a_i\leq\delta_i$
for $i>k$. If $\delta_i=0$ for some $i>k$, then the variable $x_i$
does not occur in any element of $G(L)$ because $a_i=0$. Hence 
we can replace $R$ by $K[X\setminus\{x_i\}]$. Thus we may assume 
that $\delta_i\geq 1$ for $i>k$. To polarize $L$ we use the set of
variables
$$
X_L=(\cup_{i=1}^k\{x_{i,2},\ldots, x_{i,\delta_i},x_{i,\delta_i+1},\ldots,
x_{i,a_i}\})\cup(\cup_{i=k+1}^n\{x_{i,2},\ldots,
x_{i,\delta_i}\}),
$$
where $\{x_{i,2},\ldots,x_{i,c}\}$ is the empty set if $c=0$ or $c=1$. It
is convenient to identify $x_i$ with $x_{i,1}$ for all $i$. In this 
setting the monomial $x_i^{a_i}$ polarizes to 
$(x_i^{a_i})^{\rm pol}=x_{i,2}\cdots x_{i,a_i}x_i$ for $i=1,\ldots,k$ and
the monomial $x_i^{\delta_i}$ polarizes to 
$(x_i^{\delta_i})^{\rm pol}=x_{i,2}\cdots x_{i,\delta_i}x_i$ for $i>k$. 
Let $f^{\rm pol}$ and $f_i^{\rm pol}$ be the polarizations in $R[X_L]$ of $f$
and $f_i$ (see Example~\ref{sep12-17}). By Lemma~\ref{sep10-17} one has
\begin{equation}\label{sep9-17}
\depth(R[X_L]/L^{\rm pol})=|X_L|+\depth(R/L)=\sum_{i=1}^k(a_i-1)+
\sum_{i=k+1}^n(\delta_i-1)+\depth(R/L).
\end{equation}

Next we relate the depth of $R[X_L]/(f_1^{\rm pol},\ldots,f_r^{\rm pol}))$ to
the depth of $R/I$. For this consider the polynomial ring 
$R'=K[X']$, where
$X'=(X\setminus\{x_i\}_{i=1}^k)\cup\{x_{1,\delta_i+1}\}_{i=1}^k$, 
and let $f_i'$ be the polynomial of $R'$ obtained from $f_i$ by
replacing $x_i$ with $x_{i,\delta_i+1}$ for $i=1,\ldots,k$. If $I'$ 
is the ideal of $R'$ generated by $f_1',\ldots,f_r'$, then
$K[X]/I$ and $K[X']/I'$ are isomorphic and have the same depth.
By polarizing $f_i'$ with respect to 
$$
X_{I'}=\cup_{i=1}^n\{x_{i,2},\ldots,x_{i,\delta_i}\}
$$
we obtain that $(f_i')^{\rm pol}$ is equal to $f_i^{\rm pol}$, the
polarization of $f_i$ with respect to $X_L$. The 
full polarization of $I'$ is
$(I')^{\rm pol}=((f_1')^{\rm pol},\ldots,(f_r')^{\rm pol})$. Therefore, by
Lemma~\ref{sep10-17}, one has
\begin{eqnarray}
\depth(R[X_L]/(f_1^{\rm pol},\ldots,f_r^{\rm pol}))&=&
\depth(R[X_L]/((f_1')^{\rm pol},\ldots,(f_r')^{\rm pol})),\label{sep6-17}\\
\depth(R'[X_{I'}]/((f_1')^{\rm pol},\ldots,(f_r')^{\rm pol}))&=&|X_{I'}|+\depth(R'/I')=
|X_{I'}|+\depth(R/I).\label{sep7-17}
\end{eqnarray}

As $|X\cup X_L|=\sum_{i=1}^ka_i+\sum_{i=k+1}^n\delta_i$ and $|X'\cup
X_{I'}|=\sum_{i=1}^n\delta_i$, we get  
$$
|(X\cup X_L)\setminus(X'\cup X_{I'})|=\sum_{i=1}^k(a_i-\delta_i),
$$
that is, the number of variables of $R[X_L]$ that do not occur in
$R'[X_{I'}]$ is $\sum_{i=1}^k(a_i-\delta_i)$. 
Therefore from Eqs.~(\ref{sep6-17}) and (\ref{sep7-17}), and using that
$|X_{I'}|=\sum_{i=1}^n(\delta_i-1)$, we get
\begin{eqnarray}
\depth(R[X_L]/(f_1^{\rm pol},\ldots,f_r^{\rm pol}))&=&\sum_{i=1}^k(a_i-\delta_i)+
\depth(R'[X_{I'}]/((f_1')^{\rm pol},\ldots,(f_r')^{\rm pol}))\nonumber\\
&=&\sum_{i=1}^k(a_i-1)+\sum_{i=k+1}^n(\delta_i-1)+\depth(R/I).\label{sep5-17}
\end{eqnarray}

Using Eqs.~(\ref{sep9-17}) and (\ref{sep5-17}) the required equality
follows. 

Subcase (i.b): $a_i\leq \delta_i$ for all $i$. This case follows adapting the 
arguments of Subcase (i.a), noting that $k=0$ in this case.

(ii): To prove this part we keep the notation of part (i). 

Subcase (ii.a):  Assume that $a_i>\delta_i$ for some $i$. The first 
equality follows at once from Lemma~\ref{sep10-17}. 
As $R'[X_{I'}]$ is a subring of $R[X_L]$, the regularity of $(I')^{\rm pol}R'[X_{I'}]$
is equal to that of $(I')^{\rm pol}R[X_L]$. Hence, by Lemma~\ref{sep10-17}, we get
$$
\reg(R[X_L]/(I')^{\rm pol})=\reg(R'[X_{I'}]/(I')^{\rm pol})=\reg(R'/I')=\reg(R/I).
$$
\quad Subcase (ii.b): $a_i\leq \delta_i$ for all $i$. This case follows adapting the 
arguments of Subcase (ii.a).
\end{proof}

The following corollary extends Proposition~\ref{square-free-case}
and Remark~\ref{sep17-17} from squarefree monomial ideals to
arbitrary monomial ideals   
using polarization.  It will be used throughout the paper
(e.g., Lemma~\ref{morey-lemma}, Theorem~\ref{winter2017},
Proposition~\ref{very-well-covered}). 
This result is later extended using Gr\"obner
bases (Corollary~\ref{feb20-18}).

\begin{corollary}\label{Caviglia-et-al-lemma-alternate-proof} Let 
$I\subset R$ be a monomial ideal, let $f$ be a monomial of degree $k$,
and let $x_{i_1},\ldots,x_{i_k}$ be a set of distinct variables
of $R$. The following hold.
\begin{enumerate}
\item[(i)] $\depth(R/(f,I))\geq\depth(R/I)-1$.
\item[(ii)] {\cite[Theorem~3.1]{Caviglia-et-al}} $\depth(R/I)\leq\depth(R/(I\colon f))$.
\item[(iii)] $\depth(R/(x_{i_1},\ldots,x_{i_k},I))\geq \depth(R/I)-k$.
\item[(iv)] {\cite[Theorem~3.1]{Caviglia-et-al}}
$\reg(R/I)\geq\reg(R/(I\colon f))$.
\item[(v)] $\reg(R/(f,I))\leq\reg(R/I)+k-1$. 
\item[(vi)] $\depth(R/I)=\depth(R/(I\colon f))$ or $\depth(R/I)=
\depth(R/(f,I))$.
\item[(vii)] \cite{Caviglia-et-al,Dao-Huneke-Schweig} If $k=1$, 
then $\reg(R/I)=\reg(R/(I\colon f))+1$ or
$\reg(R/I)=\reg(R/(f,I))$.
\end{enumerate}
\end{corollary}

\begin{proof} If $I$ and $f$ are squarefree, the result holds true. Indeed,
by Lemma~\ref{Caviglia-et-al-square-free}, one has 
the inequality $\depth(R/(I\colon f))\geq\depth(R/I)$. Then by
Proposition~\ref{square-free-case} and Remark~\ref{sep17-17} the
statements all hold.   
To show the general case we will use the polarization technique. 

(i) One
may assume that $f\notin I$.  
We set $G(I)=\{f_1,\ldots,f_r\}$ and 
$L=(f,I)$. Let $X_L$ be the set of new variables needed to polarize
$L$ and let $f^{\rm pol}$, $f_i^{\rm pol}$ be the polarizations in
$R[X_L]$ of $f$, $f_i$, respectively. As these polarizations are
squarefree, by Proposition~\ref{square-free-case} one has
\begin{eqnarray*}
\depth(R[X_L]/(f^{\rm pol},f_1^{\rm pol},\ldots,f_r^{\rm pol}))\geq
\depth(R[X_L]/(f_1^{\rm pol},\ldots,f_r^{\rm pol}))-1,
\end{eqnarray*}
where $L^{\rm pol}=(f^{\rm pol},f_1^{\rm pol},\ldots,f_r^{\rm pol})$. Hence, by 
Proposition~\ref{pepe-morey-vila-vivares}, $\depth(R/L)\geq\depth(R/I)-1$.

(ii): According to Lemma~\ref{depth-lemma-duality} parts (ii) and (i)
are equivalent.

(iii): It follows from part (i) using induction on $k$.

(iv)--(v): Setting $N=(R/(I\colon f))[-k]$, $M=R/I$ and $L=R/(I,f)$,
and noticing that $\reg(N)=k+\reg(R/(I\colon f))$, from 
Lemma~\ref{reg-lemma-special-k} it follows that (iv) and (v) 
are equivalent. Since $f^{\rm pol},f_1^{\rm pol},\ldots,f_r^{\rm pol}$ are
squarefree, by Proposition~\ref{square-free-case} one has
$$ 
\reg(R[X_L]/L^{\rm pol})-\reg(R[X_L]/(f_1^{\rm pol},\ldots,f_r^{\rm pol}))\leq
k-1.
$$
\quad Hence, by Proposition~\ref{pepe-morey-vila-vivares}, one 
has $\reg(R/L)-\reg(R/I)\leq
k-1$. Thus (v) and (iv) hold.

(vi): This condition is equivalent to (i). This follows applying
Lemma~\ref{depth-lemma-duality} to the exact sequence
\begin{equation*}
0\longrightarrow R/(I\colon
f)[-k]\stackrel{f}{\longrightarrow}R/I\longrightarrow
R/(I,f)\longrightarrow 0.
\end{equation*}

(vii): Recall that $\reg(R/(I\colon f))[-k]=k+\reg(R/(I\colon f))$. If
$k=1$, using Lemma~\ref{reg-lemma-special} it follows that conditions 
(vii) and (iv) are equivalent.
\end{proof}

\begin{corollary}\label{feb20-18} Let $I\subset R$ be
a monomial ideal and let $f$ be a homogeneous polynomial of degree
$k$. If there exists a monomial order $\prec$ on $R$ such that ${\rm
in}_\prec(I,f)=I+({\rm in}_\prec(f))$, then  
\begin{itemize}
\item[\rm(a)] $\depth(R/(I\colon f))\geq\depth(R/I)$,
\item[\rm(b)] $\reg(R/(I,f))\leq\reg(R/I)+k-1$, and  
\item[\rm(c)] $\reg(R/I)\geq\reg(R/(I\colon f))$.
\end{itemize}
\end{corollary}

\begin{proof} (a): We proceed by contradiction assuming that
$\depth(R/I)>\depth(R/(I\colon f))$. From the exact sequence 
\begin{equation*}
0\longrightarrow R/(I\colon
f)[-k]\stackrel{f}{\longrightarrow}R/I\longrightarrow
R/(I,f)\longrightarrow 0,
\end{equation*}
using the depth lemma \cite[Lemma~2.3.9]{monalg-rev} and the fact 
that the depth of $R/(I,f)$ is greater than or equal to 
the depth of $R/{\rm in}_\prec(I,f)$
\cite[Theorem~3.3.4(d)]{Herzog-Hibi-book}, we get  
$$
\depth(R/(I\colon f))=\depth(R/(I,f))+1\geq\depth(R/(I+{\rm
in}_\prec(f)))+1.
$$
\quad 
By Corollary~\ref{Caviglia-et-al-lemma-alternate-proof}(i), we have $\depth(R/(I+{\rm
in}_\prec(f)))\geq \depth(R/I)-1$.
Hence we obtain
$\depth(R/(I\colon f))\geq\depth(R/I)$, a contradiction.

(b): Using that the regularity of $R/(I,f)$ is less than or equal to 
the regularity of $R/{\rm in}_\prec(I,f)$
\cite[Theorem~3.3.4(c)]{Herzog-Hibi-book} and
Corollary~\ref{Caviglia-et-al-lemma-alternate-proof}(v), we get  
$$
\reg(R/(I,f))\leq\reg(R/{\rm in}_\prec(I,f))=\reg(R/(I+{\rm
in}_\prec(f)))\leq \reg(R/I)+k-1.
$$
\quad (c): Setting $N=R/(I\colon f)[-k]$, $M=R/I$, and $L=R/(I,f)$, we
proceed by contradiction assuming $\reg(R/(I\colon f))>\reg(R/I)$,
that is, $\reg(N)\geq\reg(M)+k+1$. On the other hand, by part (b), one
has $\reg(L)\leq\reg(N)-2$. According to \cite[Corollary~20.19]{Eisen}(a), one
has either $\reg(N)\leq\reg(M)$ or $\reg(N)\leq\reg(L)+1$, a
contradiction. 
\end{proof}

The next example illustrates the polarizations used in the proof 
of Proposition~\ref{pepe-morey-vila-vivares}. For convenience we use
the notation of that proof. 

\begin{example}\label{sep12-17} Let $f=x_1^3x_2^3$, $f_1=x_1^2x_3$,
$f_2=x_1x_3^2$, $f_3=x_2^2x_3$ be monomials in the polynomial ring 
$R=K[x_1,x_2,x_3]$ and set $I=(f_1,f_2,f_3)$ and $L=(f,I)$. Setting 
$$
f^{\rm pol}=x_{1,2}x_{1,3}x_1x_{2,2}x_{2,3}x_2,\, f_1^{\rm pol}=x_{1,2}x_{1,3}x_{3,2},\, 
f_2^{\rm pol}=x_{1,2}x_{3,2}x_3,\, f_3^{\rm pol}=x_{2,2}x_{2,3}x_{3,2},
$$
and $X_L=\{x_{1,2},x_{1,3}\}\cup\{x_{2,2},x_{2,3}\}\cup\{x_{3,2}\}$, 
the full polarization of $L$ is   
$$
L^{\rm pol}=(f^{\rm pol},\,f_1^{\rm pol},\,f_2^{\rm pol},f_3^{\rm pol})\subset R[X_L].
$$
\quad Making the change of variables $x_1\rightarrow x_{1,3}$,
$x_2\rightarrow x_{2,3}$ in $I$ and setting 
$$
f_1'=x_{1,3}^2x_3,\, f_2'=x_{1,3}x_3^2,\, f_3'=x_{2,3}^2x_3,\, 
I'=(f_1',\,f_2',\,f_3'),
$$
$X_{I'}=\{x_{1,2}\}\cup\{x_{2,2}\}\cup\{x_{3,2}\}$,
$R'=K[x_{1,3},x_{2,3},x_3]$, the full
polarization of $I'$ is
$$
(I')^{\rm pol}=((f_1')^{\rm pol},(f_2')^{\rm pol},(f_3')^{\rm pol})\subset
R'[X_{I'}],
$$ 
where $(f_1')^{\rm pol}=x_{1,2}x_{1,3}x_{3,2},
\,(f_2')^{\rm pol}=x_{1,2}x_{3,2}x_3,\,
(f_3')^{\rm pol}=x_{2,2}x_{2,3}x_{3,2}$. Thus $f_i^{\rm pol}$ is equal to 
$(f_i')^{\rm pol}$ 
for $i=1,2,3$. Setting $X_I=\{x_{1,2},x_{2,2},x_{3,2}\}$ the full
polarization of $I$ is generated by the monomials 
$x_{1,2}x_1x_{3,2},\,x_{1,2}x_{3,2}x_3,\, x_{2,2}x_2x_{3,2}$.
\end{example}

\section{Depth and regularity locally at each variable}\label{depths}

In this section we use polarization to study the behavior of the depth and regularity of
a monomial ideal locally at each variable when lowering the top degree.  

Let $R=K[x_1,\ldots,x_n]$ be a polynomial ring over a field $K$, 
let $I$ be a monomial ideal of $R$ and let $x_i$ be a fixed variable that occurs 
in $G(I)$. Given a monomial
$x^a=x_1^{a_1}\cdots x_n^{a_n}$, we set $\deg_{x_i}(x^a)=a_i$. Consider the
integer 
$$
q:=\max\{\deg_{x_i}(x^a)\vert\, x^a\in G(I)\}, 
$$
and the corresponding set $\mathcal{B}_i:=\{x^a\vert\,
\deg_{x_i}(x^a)=q\}\cap G(I)$. That is, $\mathcal{B}_i$ is the set of all monomial
of $G(I)$  of highest degree in $x_i$. Setting 
$$
\mathcal{A}_i:=\{x^a\vert\, \deg_{x_i}(x^a)<q\}\cap
G(I)=G(I)\setminus\mathcal{B}_i,
$$
$p:=\max\{\deg_{x_i}(x^a)\vert\, x^a\in\mathcal{A}_i\}$ and 
$L:=(\{x^a/x_i\vert\,x^a\in\mathcal{B}_i\}\cup\mathcal{A}_i\})$, we are
interested in comparing the depth (resp. regularity) of $R/I$ with the
depth (resp. regularity) of $R/L$. 

One of the main results of this section shows that the depth is
locally non-decreasing at each variable $x_i$ when lowering the top degree:

\begin{theorem}\label{morey-vila-oaxaca-2017}
Let $I$ be a monomial ideal of $R$ and let $x_i$ be a variable. The following hold.
\begin{enumerate}
\item[(a)] If $p\geq 1$ and $q-p\geq 2$, then
$\depth(R/I)=\depth(R/L)$.  
\item[(b)] If $p\geq 0$ and $q-p=1$, then $\depth(R/L)\geq\depth(R/I)$. 
\item[(c)] If $p=0$ and $q\geq 2$, then 
$
\depth(R/I)=\depth(R/(\{x^a/x_i^{q-1}\vert\,x^a\in\mathcal{B}_i\}\cup\mathcal{A}_i\})).
$
\end{enumerate}
\end{theorem}

\begin{proof} (a): To simplify notation we set $i=1$. We may assume that
$G(I)=\{f_1,\ldots,f_r\}$, where $\{f_1,\ldots,f_m\}$ is the
set of all elements of $G(I)$ that contain $x_1^q$ and
$\{f_{m+1},\ldots,f_s\}$ is the set of all elements of $G(I)$ that 
contain some positive power $x_1^\ell$ of $x_1$ for some $1\leq \ell<q$. Making a
partial polarization of $x_1^q$ with respect to the new variables
$x_{1,2},\ldots,x_{1,q-1}$ \cite[p.~203]{monalg-rev}, gives that $f_j$ polarizes to 
$f_j^{\rm pol}=x_{1,2}\cdots x_{1,q-1}x_1^2f_j'$ for $j=1,\ldots,m$,
where $f_1',\ldots,f_m'$ are monomials that do not contain $x_1$ and
$f_j=x_1^qf_j'$ for $j=1,\ldots,m$.  
Hence, using that $q-p\geq 2$, one has the partial polarization
$$
I^{\rm pol}=(x_{1,2}\cdots x_{1,q-1}x_1^2f_1',\ldots,x_{1,2}\cdots
x_{1,q-1}x_1^2f_m',f_{m+1}^{\rm pol},\ldots,f_s^{\rm
pol},f_{s+1},\ldots,f_r),
$$
where $f_{m+1}^{\rm pol},\ldots,f_s^{\rm pol}$ do not contain $x_1$
and $I^{\rm pol}$ is an ideal of $R^{\rm
pol}=R[x_{1,2},\ldots,x_{1,q-1}]$. 
On the other hand, from the equality
$$
G(L)=\{f_1/x_1,\ldots,f_m/x_1,f_{m+1},\ldots,f_r\},
$$
one has the partial polarization
$$
L^{\rm pol}=(x_{1,2}\cdots x_{1,q-1}x_1f_1',\ldots,x_{1,2}\cdots
x_{1,q-1}x_1f_m',f_{m+1}^{\rm pol},\ldots,f_s^{\rm
pol},f_{s+1},\ldots,f_r).
$$

By making the substitution
$x_1^{2}\rightarrow x_1$ in
each element of $G(I^{\rm pol})$ this will not affect the depth of
$R^{\rm pol}/I^{\rm pol}$ (see \cite[Lemmas~3.3 and
3.5]{lattice-dim1}). Thus 
$$
q-2+\depth(R/I)=
\depth(R^{\rm pol}/I^{\rm pol})=\depth(R^{\rm pol}/L^{\rm
pol})=q-2+\depth(R/L),
$$
and consequently $\depth(R/I)=\depth(R/L)$.

(b): To simplify notation we set $i=1$. Assume $p=0$, then $q=1$. Note
that the ring $R/L$ is equal
to $R/(I\colon x_1)$. Hence, by
Corollary~\ref{Caviglia-et-al-lemma-alternate-proof},  
its depth is greater than or equal to 
$\depth(R/I)$. Thus we may assume that $p\geq 1$. 
We may also assume that
$G(I)=\{f_1,\ldots,f_r\}$, where $f_1,\ldots,f_m$ is the
set of all elements of $G(I)$ that contain $x_1^q$, and  
$f_{m+1},\ldots,f_t$ is the set of all elements of $G(I)$ that 
contain $x_1^{q-1}$ but not $x_1^{q}$, and $f_{t+1},\ldots,f_s$ is the
set of all elements of $G(I)$ that contain some power $x_1^\ell$, with
$1\leq\ell<q-1$, but not $x_1^{\ell+1}$. 
Let $R'$ be the polynomial ring 
$K[x_{1,q},x_2,\ldots,x_n]$, with $x_{1,q}$ a new variable, and let $L'$
be the ideal of $R'$ obtained from $L$ by making the change of
variable $x_1\rightarrow x_{1,q}$ in each element of $G(L)$. Clearly 
\begin{equation*}
\depth(R/L)=\depth(R'/L')=\depth(R'[x_1]/L')-1.
\end{equation*}

The partial polarization of $I$ with
respect to $x_1$ using the variables $x_{1,2},\ldots,x_{1,q}$ is given
by
\begin{eqnarray*}
I^{\rm pol}=(x_{1,2}\cdots x_{1,q}x_1f_1',\ldots,x_{1,2}\cdots
x_{1,q}x_1f_m',& &\\ x_{1,2}\cdots x_{1,q}f_{m+1}',\ldots,x_{1,2}\cdots
x_{1,q}f_{t}',& &\\
f_{t+1}^{\rm pol},\ldots,f_s^{\rm
pol},f_{s+1},\ldots,f_r),
\end{eqnarray*}
where $f_1',\ldots,f_t',f_{t+1}^{\rm pol},\ldots,f_s^{\rm
pol},f_{s+1},\ldots,f_r$ do not contain $x_1$
and $I^{\rm pol}$ is an ideal of the ring $R^{\rm
pol}=R[x_{1,2},\ldots,x_{1,q}]$. Therefore
\begin{eqnarray*}
(I^{\rm pol}\colon x_1)=(x_{1,2}\cdots x_{1,q}f_1',\ldots,x_{1,2}\cdots
x_{1,q}f_m',\ \ \ \ \ \ \ \ \ \ \ \ \ \ \ \ \ \ \ \ \ \ \ \ \ \ \ \ \ \ 
& &\\ x_{1,2}\cdots x_{1,q}f_{m+1}',\ldots,x_{1,2}\cdots
x_{1,q}f_{t}',f_{t+1}^{\rm pol},\ldots,f_s^{\rm
pol},f_{s+1},\ldots,f_r).
\end{eqnarray*}

The following is a generating set for $L'$, which is not necessarily minimal:
\begin{eqnarray*}
L'=(x_{1,q}^{q-1}f_1',\ldots,x_{1,q}^{q-1}f_m',
x_{1,q}^{q-1}f_{m+1}',\ldots,x_{1,q}^{q-1}f_{t}',& &\\
x_{1,q}^{a_{t+1}}f_{t+1}',\ldots,x_{1,q}^{a_s}f_s',f_{s+1},\ldots,f_r),
\end{eqnarray*}
where $1\leq a_i<q-1$ for $i=t+1,\ldots,s$. 
Hence, it is seen that, $(I^{\rm pol}\colon x_1)$ is equal to
$(L')^{\rm pol}$, 
the polarization of $L'$ with respect to the variable $x_{1,q}$ using
the variables $x_{1,2},\ldots,x_{1,q-1}$. Therefore, using
Lemma~\ref{Caviglia-et-al-square-free}, we get
\begin{eqnarray*}
(q-1)+\depth(R/L)&=&1+((q-2)+\depth(R'/L'))\\
&=&1+\depth((R')^{\rm
pol}/(L')^{\rm pol})= \depth((R'[x_1])^{\rm
pol}/(L')^{\rm pol})\\ 
&=&\depth(R^{\rm pol}/(L')^{\rm pol})=\depth(R^{\rm pol}/(I^{\rm pol}\colon x_1))\\
&\geq&
\depth(R^{\rm pol}/I^{\rm pol})=(q-1)+\depth(R/I).
\end{eqnarray*}
Thus $\depth(R/L)\geq\depth(R/I)$.

(c): It suffices to notice that by making the substitution
$x_i^{q}\rightarrow x_i$ in
each element of $G(I)$ this will not affect the depth of $R/I$ 
(see \cite[Lemmas~3.3 and 3.5]{lattice-dim1}).
\end{proof}

Let  $\mathcal{D}$ be a {\it vertex-weighted digraph\/}, that is, 
$\mathcal{D}$ consists of a finite set 
$V({\mathcal D})=\{x_1,\ldots,x_n\}$ of vertices, a prescribed collection
$E({\mathcal D})$ of ordered pairs of distinct points called {\it
edges\/} or {\it arrows}, and $\mathcal{D}$ is endowed with a function $d\colon 
V(\mathcal{D})\rightarrow\mathbb{N}_+$, where
$\mathbb{N}_+:=\{1,2,\ldots\}$. The weight $d(x_i)$ of $x_i$ is denoted simply by $d_i$. 
The {\it edge ideal\/} of $\mathcal{D}$, denoted 
$I(\mathcal{D})$, is the ideal of $R$ given by
$$
I(\mathcal{D}):=(x_ix_j^{d_j}\, \vert\, (x_i,x_j)\in E(\mathcal{D})).
$$

Edge ideals of vertex-weighted digraphs occur in the theory of
Reed-Muller-type codes as initial ideals of vanishing ideals of 
projective spaces over a finite field
\cite{Ha-Lin-Morey-Reyes-Vila,hilbert-min-dis,sorensen}.

\begin{corollary}{\cite[Corollary~6]{cm-oriented-trees}}\label{oriented-graphs}
Let $I=I(\mathcal{D})$ be the edge ideal of a
vertex-weighted digraph with vertices $x_1,\ldots,x_n$ 
and let $d_i$ be the weight of $x_i$. If $\mathcal{U}$ is the digraph
obtained from $\mathcal{D}$ by assigning weight $2$ to every vertex $x_i$
with $d_i\geq 2$, then $I$ is Cohen--Macaulay if and only 
if $I(\mathcal{U})$ is Cohen--Macaulay. 
\end{corollary}

\begin{proof} By applying Theorem~\ref{morey-vila-oaxaca-2017} to each
vertex $x_i$ of $\mathcal{D}$ of weight at least $3$, we obtain 
that the depth of $R/I(\mathcal{D})$ is equal to
the depth of $R/I(\mathcal{U})$. Since $I(\mathcal{D})$ and $I(\mathcal{U})$ have the
same height, then $I(\mathcal{D})$ is Cohen--Macaulay if and only if 
$I(\mathcal{U})$ is Cohen--Macaulay. 
\end{proof}

\begin{corollary}{\cite{herzog-takayama-terai}}\label{herzog-takayama-terai-theo}
If $I$ is a monomial ideal, then
$\depth(R/{\rm rad}(I))\geq\depth(R/I)$. In particular if $I$ is
Cohen--Macaulay, then ${\rm rad}(I)$ is Cohen--Macaulay.
\end{corollary}

\begin{proof} It follows by applying
Theorem~\ref{morey-vila-oaxaca-2017} to every vertex 
$x_i$  as many times as necessary. 
\end{proof}

As a consequence if $I$ is squarefree, then 
$\depth(R/I)\geq \depth (R/I^k)$ for all $k\geq 1$.

\begin{remark} Let $L\subset R$ be a monomial ideal. If $x_i^k$ is in
$G(L)$ for some $k\geq 1$, $1\leq i\leq n$ and $L'$ is the ideal of
$R$ generated by all elements of $G(I)$ that do not contain $x_i$,
then $(L,x_i)=(L',x_i)$ and by a repeated application of
Theorem~\ref{morey-vila-oaxaca-2017} one has 
$$
\depth(R/L)\leq\depth(R/(L',x_i))=\depth(R/L')-1.
$$
\end{remark}

Before proving an analog of Theorem~\ref{morey-vila-oaxaca-2017} for
regularity, we first provide a basic fact regarding the effect of a
change of variables on the resolution of an ideal. 

\begin{lemma}\label{powers-and-resolutions}
Let $I$ be a homogeneous ideal of $R$, let $d_1$ be a positive
integer, and define $\phi\colon R\rightarrow
R$ by $\phi(x_1)=x_1^{d_1}$ and $\phi(x_i)=x_i$ for $2\leq i \leq n$.
If $\phi(I)$ is homogeneous, then a minimal resolution of $\phi(I)$
over $R$ can be obtained by applying $\phi$ to a minimal resolution
of $I$. Moreover, the $($non-graded$)$ Betti numbers of $I$ and $\phi(I)$
will be equal and $\reg(\phi(I)) \geq \reg I$.      
\end{lemma}

\begin{proof}
Define $S=K[x_1,\ldots,x_n]$ to be a polynomial ring with the
non-standard grading $d(x_1)=d_1$ and $d(x_i)=1$ for 
$2 \leq i \leq n$. Note that the map $\phi$ factors through $S$. Write $\phi = \psi
\sigma$, where $\sigma\colon R \rightarrow S$ is given by $\sigma(x_i)=x_i$
for all $i$ and $\psi\colon S \rightarrow R$ is given by
$\psi(x_1)=x_1^{d_1}$ and $\psi(x_i)=x_i$ for $2 \leq i \leq n$.
Then, by assumption, $I$ is a homogeneous ideal of $R$ and
$\sigma(I)$ is again homogeneous in $S$. Applying $\sigma$ to a
minimal resolution of $I$ yields a minimal resolution of $\sigma(I)$,
where the modules and maps are unchanged except that the degrees of
some of the maps, and thus the shifts in the resolution, may have
increased, showing $\reg(\sigma(I)) \geq \reg (I)$. Now the map
$\psi$ is precisely the map used in \cite[Lemma 3.5 and Theorem
3.6(b)]{lattice-dim1}. The result follows from combining these results.             
\end{proof}

\begin{lemma}\label{morey-lemma} 
Let $I$ and $J$ be monomial ideals of $R$ and let $x_i$
be a variable. If $(I\colon x_i)=J$ and $(I,x_i)=(J,x_i)$, then 
\begin{itemize}
\item[(i)] $\reg(R/J)\leq\reg(R/I)\leq\reg(R/J)+1$, and 
\item[(ii)] $\depth(R/J)-1\leq\depth(R/I)\leq\depth(R/J)$.
\end{itemize}
\end{lemma}

\begin{proof} (i): By
Corollary~\ref{Caviglia-et-al-lemma-alternate-proof}(v), we have 
$\reg(R/(I,x_i))\leq \reg(R/I)$ and $\reg(R/(J,x_i))\leq \reg(R/J)$, 
and by Corollary~\ref{Caviglia-et-al-lemma-alternate-proof}(vii), we have either
$\reg(R/I)=\reg(R/(I\colon x_i))+1= \reg(R/J)+1$ or
$\reg(R/I)=\reg(R/(I,x_i)) = \reg(R/(J,x_i)) \leq \reg(R/J)$. In the
latter case one has $\reg(R/I)=\reg(R/J)$ because by 
Corollary~\ref{Caviglia-et-al-lemma-alternate-proof}(iv), one 
has $\reg(R/J)\leq\reg(R/I)$. Combining these facts yields $\reg(R/J) \leq \reg (R/I) \leq
\reg(R/J)+1$. 

(ii): By Corollary~\ref{Caviglia-et-al-lemma-alternate-proof}(vi), we
have either
$\depth(R/I)=\depth(R/J)$ or $\depth(R/I)=\depth(R/(I,x_i))$. In the
latter case one has
$$
\depth(R/J)\geq\depth(R/I)=\depth(R/(I,x_i))=\depth(R/(J,x_i))
\geq \depth(R/J)-1
$$
because by parts (ii) and (i) of
Corollary~\ref{Caviglia-et-al-lemma-alternate-proof} 
one has the inequalities $\depth(R/J)\geq\depth(R/I)$ and
$\depth(R/(J,x_i))\geq
\depth(R/J)-1$, respectively. 
\end{proof}

Using the notation introduced for
Theorem~\ref{morey-vila-oaxaca-2017} we are now able to control
regularity when lowering the degrees of the generators of a monomial ideal.  

\begin{theorem}\label{morey-vila-oaxaca-2017-reg}
Let $I$ be a monomial ideal and let $L'$ be the ideal 
$(\{x^a/x_i^{q-1}\vert\,x^a\in\mathcal{B}_i\}\cup\mathcal{A}_i\})$,
where $x_i$ is a variable.
The following hold.
\begin{enumerate}
\item[(a)] If $p\geq 1$ and $q-p\geq 2$, then
$\reg(R/L) \leq \reg(R/I)\leq \reg(R/L)+1$.  
\item[(b)] If $p\geq 0$ and $q-p=1$, then $\reg(R/L)\leq\reg(R/I)$. 
\item[(c)] If $p=0$ and $q\geq 2$,  then 
$\reg(R/L')\leq \reg(R/I) \leq\reg(R/L')+q-1$.
\end{enumerate}
\end{theorem}

\begin{proof}
(a): As in Theorem~\ref{morey-vila-oaxaca-2017}, we assume $i=1$.
Forming a partial polarization of $x_1^q$ with respect to new
variables $x_{1,2}, \ldots , x_{1,q-1}$ will not change the
regularity by Lemma~\ref{sep10-17} (d). By the same argument, forming
a full polarization of $x_2, \ldots, x_n$ will also not change the
regularity. Thus we may assume that $I=(x_1^2h_1, \ldots, x_1^2h_m,
h_{m+1}, \ldots , h_r)$ and $L=(x_1h_1, \ldots ,x_1h_m, h_{m+1},
\ldots, h_r)$ where $h_j$ are squarefree monomials and $x_1$ does not divide
$h_j$ for all $j$. Note that $(I,x_1)=(L,x_1)$ and $(I\colon x_1)=L$. Thus,
by Lemma~\ref{morey-lemma}, we have $\reg(R/I) = \reg(R/L)+1$ or $\reg(R/I) =
\reg (R/L)$ as claimed.    

(b): This part follows from the proof of
Theorem~\ref{morey-vila-oaxaca-2017}(b) and Lemma~\ref{sep10-17}(d). 

(c): We proceed by induction on $q\geq 2$. There are monomials $h_1,\ldots,h_r$ not 
containing $x_1$ such that
$$
I=(x_1^qh_1,\ldots,x_1^qh_m,h_{m+1},\ldots,h_r)\ \mbox{ and }\ 
L=(x_1^{q-1}h_1,\ldots,x_1^{q-1}h_m,h_{m+1},\ldots,h_r).
$$
\quad Note that $(I,x_1)=(L,x_1)$ and $L=(I\colon x_1)$. Then,
applying Lemma~\ref{morey-lemma} to $I$ and $L$, one has $\reg(R/L) \leq \reg (R/I) \leq
\reg(R/L)+1$. In particular the required inequality holds for $q=2$.
If $q>3$, applying induction to $L$, the inequality follows.         
\end{proof}

\begin{corollary}\label{feb21-18} Let $I$ be a monomial ideal of $R$
and let $J$ be
its radical. The following hold.
\begin{itemize}
\item[(i)] \cite{ravi} $\reg(R/J)\leq\reg(R/I)$.
\item[(ii)] If $I$ is Cohen--Macaulay, then 
$a(R/J)\leq a(R/I)$, where $a(\cdot)$ is the $a$-invariant. 
\end{itemize}
\end{corollary}

\begin{proof} (i): It follows by applying
Theorem~\ref{morey-vila-oaxaca-2017-reg} to every vertex 
$x_i$  as many times as necessary. 

(ii): By Corollary~\ref{herzog-takayama-terai-theo}, $J$ is
Cohen--Macaulay. Hence, by \cite[Corollary~B.4.1]{Vas1}, one has 
$a(M)={\rm reg}(M)-{\rm depth}(M)$ for $M=R/I$ and $M=R/J$. As 
$\dim(R/I)=\dim(R/J)=\depth(R/I)=\depth(R/J)$, the inequality follows
from part (i).
\end{proof}

\begin{remark} 
Let $I\subset R$ be a monomial ideal and let $f$ be a monomial which
is a non-zero divisor of $R/I$. Then $\reg(R/fI)=\reg(R/I)+\deg(f)$ 
and $\reg(R/(I,f))=\reg(R/I)+\deg(f)-1$. This follows from 
Proposition~\ref{additivity-cm-square}. Thus the upper bound of 
 Theorem~\ref{morey-vila-oaxaca-2017-reg}(c) is tight. 
\end{remark}

\begin{example} The ideals $I=(x_1^2x_2x_3^2,x_3^2x_4,x_4^3x_5)$ and 
$J=(x_1x_2x_3^2,x_3^2x_4,x_4^3x_5)$ have regularity $5$. Thus 
the lower bound of Theorem~\ref{morey-vila-oaxaca-2017-reg}(c)
is also tight.
\end{example}

\begin{example} The ideals
$I=(x_1^7x_2x_3^2,x_1^7x_5^3,x_1^6x_3^2x_4,x_2x_5^7)$, 
$L=(x_1^6x_2x_3^2,x_1^6x_5^3,x_1^6x_3^2x_4,x_2x_5^7)$ 
have regularity $16$ and $13$, respectively. Thus in 
Theorem~\ref{morey-vila-oaxaca-2017-reg}(b), $\reg(R/L)+1$ is not an
upper bound for $\reg(R/I)$.
\end{example}

\section{Edge ideals of clutters with non-increasing depth}

Let $\mathcal{C}$ be a clutter with vertex set $X=\{x_1,\ldots,x_n\}$
and let $\{x^{v_1},\ldots,x^{v_r}\}$ be the minimal generating set of
$I(\mathcal{C})$. The matrix $A$ whose column 
vectors are $v_1^\top,\ldots,v_r^\top$ is called the {\it incidence
matrix\/} of $\mathcal C$. 
The {\it set covering polyhedron\/} of ${\mathcal C}$ is given by:
$$
\mathcal{Q}(A):=\{x\in\mathbb{R}^n\vert\, x\geq 0;\, xA\geq{\mathbf 1}\},
$$
where $\mathbf{1}=(1,\ldots,1)$. The rational polyhedron $\mathcal{Q}(A)$ 
is called {\it integral\/} if it has only integral vertices. A clutter
is called {\it uniform\/} (resp. {\it unmixed\/}) if all its edges
(resp. minimal vertex
covers) have the same cardinality. A clutter is {\it ideal\/} if its
set covering polyhedron is integral \cite{cornu-book}. 

\begin{definition}\rm A clutter $\mathcal{C}$, with incidence matrix
$A$, has the {\it max-flow
min-cut\/} (MFMC) property if both sides 
of the LP-duality equation
\begin{equation}\label{jun6-2-03}
{\rm min}\{\langle \alpha,x\rangle \vert\, x\geq 0; xA\geq\mathbf{1}\}=
{\rm max}\{\langle y,{1}\rangle \vert\, y\geq 0; Ay\leq\alpha\} 
\end{equation}
have integral optimum solutions $x,y$ for each nonnegative integral vector $\alpha$. 
\end{definition}

\begin{definition}\label{symbolic-power-def}\rm 
Let $I$ be a squarefree monomial ideal of $R$ and 
let $\mathfrak{p}_1,\ldots,{\mathfrak p}_r$ be  
the associated primes of $I$. Given an integer $k\geq 1$, we define 
the $k$-th {\it symbolic power} of 
$I$ to be the ideal 
$$
I^{(k)}:=\bigcap_{i=1}^r(I^kR_{\mathfrak{p}_i}\cap R)={\mathfrak
p}_1^{k}\cap\cdots\cap {\mathfrak p}_r^{k}.
$$
\end{definition}
\quad An ideal $I$ of $R$ is 
called {\em normally torsion-free} if ${\rm Ass}(R/I^k)$ is contained in ${\rm Ass}(R/I)$ for all 
$k\geq 1$. Notice that if $I$ is a squarefree monomial ideal,
then $I$ is normally torsion-free if and only 
if $I^k=I^{(k)}$ for all $k\geq 1$. A major result of
\cite{reesclu,clutter} shows that a
clutter $\mathcal{C}$ has the 
max-flow min-cut property if and only if $I(\mathcal{C})$ is normally
torsion-free (cf. \cite[Proposition~3.4]{normali}). 

\begin{lemma}{\rm(\cite[Lemma~5.6]{reesclu},
\cite[Lemma~2.1]{mfmc})}\label{nov19-03} If ${\mathcal C}$ is a 
uniform clutter and $\mathcal{Q}(A)$ is integral, then there exists a
minimal vertex cover of $\mathcal C$  
intersecting every edge of $\mathcal C$ in exactly one vertex. 
\end{lemma}

\begin{theorem}\label{dec10-17}\cite[Theorem 1.17]{cornu-book} 
If $\mathcal{Q}(A)$ is integral and
$B$ is the incidence matrix of the clutter $\mathcal{C}^\vee$ of
minimal vertex covers of $\mathcal{C}$, then
$\mathcal{Q}(B)$ is integral.
\end{theorem}

\begin{lemma}\label{nov19-03-dual} 
If ${\mathcal C}$ is an unmixed clutter and $\mathcal{Q}(A)$ is integral, then
there exists an edge of $\mathcal{C}$  
intersecting every minimal vertex cover of $\mathcal C$ in exactly one vertex.
\end{lemma}

\begin{proof} By duality \cite[Theorem~6.3.39]{monalg-rev} 
the minimal vertex covers of $\mathcal{C}^\vee$ (resp.
edges of $\mathcal{C}^\vee$) are the edges of $\mathcal{C}$ (resp.
minimal vertex covers of $\mathcal{C}$). 
Let $B$ be the incidence matrix of $\mathcal{C}^\vee$.
As $\mathcal{Q}(A)$ is integral and $\mathcal{C}$ is unmixed, by
Lemma~\ref{dec10-17}, $\mathcal{Q}(B)$ is also
integral and $\mathcal{C}^\vee$ is uniform. Thus 
applying Lemma~\ref{nov19-03} to $\mathcal{C}^\vee$, 
there exists a minimal vertex cover of $\mathcal{C}^\vee$  
intersecting every edge of $\mathcal{C}^\vee$ in exactly one 
vertex. Hence by duality the result follows.
\end{proof}

Let $I\subset R$ be a homogeneous ideal and let
$\mathfrak{m}=(x_1,\ldots,x_n)$ be the maximal irrelevant 
ideal of $R$.   
Recall that the {\it analytic spread\/} of $I$, denoted by $\ell(I)$, is given by 
$$
\ell(I)=\dim R[It]/\mathfrak{m}R[It].
$$ 
\quad This number satisfies ${\rm ht}(I)\leq \ell(I)\leq\dim(R)$
\cite[Corollary~5.1.4]{Vas}. 

\begin{theorem}\cite{Bur,E-H}\label{burch} ${\rm inf}_i\{{\rm
depth}(R/I^i)\}\leq\dim(R)-\ell(I)$, with equality if the associated graded ring ${\rm
gr}_I(R)$ is 
Cohen--Macaulay.
\end{theorem}

Brodmann \cite{brodmann} improved this inequality by showing that 
${\rm depth}(R/I^k)$ is constant for $k\gg 0$ and that this constant
value is bounded from above by $\dim(R)-\ell(I)$. For a generalization of these results to other
ideal filtrations see \cite[Theorem~1.1]{depth}. The constant value of
$\depth(R/I^k)$ for $k\gg 0$ is called the {\it limit depth\/} of $I$
and is denoted by $\lim_{k\rightarrow\infty}\depth(R/I^k)$. 

\begin{definition}
A homogeneous ideal $I\subset R$ has {\it non-increasing depth\/} if 
$$
\depth(R/I^k)\geq\depth(R/I^{k+1})\ \forall\ k\geq 1,
$$
and $I$ has {\it non-decreasing regularity\/} if 
$\reg(R/I^k)\leq\reg(R/I^{k+1})$ for all $k\geq 1$. 
The ideal $I$  has the {\it persistence property\/} if 
${\rm Ass}(R/I^k)\subset {\rm Ass}(R/I^{k+1})$ for
$k\geq 1$.
\end{definition}

There are some classes of monomial ideals with non-increasing depth and
non-decreasing regularity  
\cite{Caviglia-et-al,constantinescu-etal,Nguyen-Thu-Hang,
Hang-Trung,very-well-covered-non-inc}. A natural way to show these 
properties for a monomial ideal $I$ is to prove the existence of a
monomial $f$ such that $(I^{k+1}\colon f)=I^k$ for $k\geq 1$. This was
exploited in \cite{Caviglia-et-al,edge-ideals} and in \cite[Corollary~3.11]{Ha-Morey} in
connection to normally torsion-free ideals. 

\begin{theorem}{\cite[Theorem~5.1]{Caviglia-et-al}} 
If $I(\mathcal{C})$ is the edge ideal of a clutter $\mathcal{C}$ which has a good
leaf, then $I(\mathcal{C})$ has non-increasing depth and
non-decreasing regularity. 
\end{theorem}

In particular edge ideals of forests or simplicial trees 
have non-increasing depth and non-decreasing regularity. Our next
result gives another wide family of ideals with these properties. 

\begin{theorem}\label{winter2017} Let $\mathcal{C}$ be a clutter and let
$I=I(\mathcal{C})$ be its edge ideal. If $\mathcal{C}$ is unmixed and
satisfies the max-flow min-cut property, then 
\begin{itemize}
\item[(a)] $\depth(R/I^k)\geq\depth(R/I^{k+1})$ for $k\geq 1$, and 
\item[(b)] $\reg(R/I^k)\leq\reg(R/I^{k+1})$ for $k\geq 1$.
\end{itemize}
\end{theorem}

\begin{proof} Let $C_1,\ldots,C_s$ be the minimal vertex covers of
$\mathcal{C}$. If $\mathfrak{p}_i$ is the ideal of $R$ generated by 
$C_i$ for $i=1,\ldots,s$, then $\mathfrak{p}_1,\ldots,\mathfrak{p}_s$
are the minimal primes of $I$ \cite[Theorem~6.3.39]{monalg-rev}. 
As $\mathcal{C}$ has the max-flow min-cut property, by
\cite[Corollary~22.1c]{Schr}, $\mathcal{Q}(A)$ is integral. Therefore, by
Lemma~\ref{nov19-03-dual}, there exists an edge $e$ of $\mathcal{C}$
intersecting every $C_i$ in exactly one vertex. Thus
$|e\cap\mathfrak{p}_i|=1$ for $i=1,\ldots,s$. We claim that
$(I^{k+1}\colon x_e)=I^k$ for $k\geq 1$, where 
$x_e=\prod_{x_i\in e}x_i$. The 
$k$-th symbolic power of $I$ is given by 
\begin{equation}\label{dec11-17}
I^{(k)}=\mathfrak{p}_1^k\cap\cdots\cap\mathfrak{p}_s^k,
\end{equation}
and by \cite[Corollary~3.14]{clutter}, $I^k=I^{(k)}$ for $k\geq 1$.
Clearly $I^k$ is contained in $(I^{k+1}\colon x_e)$ because $x_e$ is
in $I$. To show the other inclusion take $x^a$ in $(I^{k+1}\colon
x_e)$. Fix any $1\leq i\leq s$. Then $x^ax_e$ is in
$I^{k+1}\subset\mathfrak{p}_i^{k+1}$. Thus there are
$x_{j_1},\ldots,x_{j_{k+1}}$ in $\mathfrak{p}_i$ with $j_1\leq\cdots\leq j_{k+1}$
such that
$$  
x^ax_e=x_{j_1}\cdots x_{j_{k+1}}x^b,
$$
for some $x^b$. Since $|e\cap\mathfrak{p}_i|=1$ from this equality, we get that with one
possible exception all variables that occur in $x_e$ divide $x^b$.
Thus $x^a\in\mathfrak{p}_i^k$. As $1\leq i\leq s$ was an arbitrary fixed integer,
using Eq.~(\ref{dec11-17}), we get $x^a\in I^{(k)}=I^k$. Thus
$(I^{k+1}\colon x_e)=I^k$, as claimed. To prove parts (a) and (b) note that, by
Corollary~\ref{Caviglia-et-al-lemma-alternate-proof}(ii), one has 
$$
\depth(R/I^k)=\depth(R/(I^{k+1}\colon
x_e))\geq\depth(R/I^{k+1}),
$$
and by Corollary~\ref{Caviglia-et-al-lemma-alternate-proof}(iv), one
has $\reg(R/I^k)=\reg(R/(I^{k+1}\colon
x_e))\leq\reg(R/I^{k+1}).$
\end{proof}

\begin{corollary}\label{winter2017-dual} Let $\mathcal{C}$ be a clutter and let
$J=I(\mathcal{C})^\vee$ be its ideal of covers. If $\mathcal{C}$ is uniform and
its set covering polyhedron is integral, then 
\begin{itemize}
\item[(a)] $\depth(R/J^{(k)})\geq\depth(R/J^{(k+1)})$ for $k\geq 1$, and 
\item[(b)] $\reg(R/J^{(k)})\leq\reg(R/J^{(k+1)})$ for $k\geq 1$.
\end{itemize}
\end{corollary}

\begin{proof} This follows using duality and adapting the proof of
Theorem~\ref{winter2017}.
\end{proof}

\section{Edge ideals of graphs}

Let $G$ be a graph with vertex set $V(G)=\{x_1,\ldots,x_n\}$. A
connected component of $G$ with at least two vertices is 
called non-trivial. We denote the set
of isolated vertices of $G$ by ${\rm isol}(G)$ and the number of
non-trivial bipartite components of $G$ by $c_0(G)$. The 
{\it neighbor} set of $x_i$, denoted $N_G(x_i)$, is the set 
of all $x_j\in V(G)$ such that $\{x_i,x_j\}$ is an edge of $G$.

\begin{proposition}\label{trung-limit-powers} 
Let $I(G)$ be the edge ideal of $G$. The following hold for $k\geq 1$ and $i=1,\ldots,n$.
\begin{itemize}
\item[(a)] $\depth(R/(I(G)^k\colon x_i^k))\leq\depth(R/(I(G\setminus
N_G(x_i))^k,N_G(x_i)))$. 
\item[(b)] \cite[p. 293]{monalg-rev} $(I(G)\colon x_i)=(I(G\setminus
N_G(x_i)),N_G(x_i))$.
\item[(c)] $\dim(R)-\ell(I(G))=|{\rm isol}(G)|+c_0(G)$.
\item[(d)] {\rm \cite[Theorem~4.4(1)]{Tran-Nam-Trung}} 
$\lim_{k\rightarrow\infty}\depth(R/I(G)^k)=|{\rm
isol}(G)|+c_0(G)$.  
\item[(e)] If $H=G\setminus
N_G(x_i)$, then $\lim_{k\rightarrow\infty}\depth(R/(I(H)^k,N_G(x_i)))=|{\rm
isol}(H)|+c_0(H)$.
\end{itemize}
\end{proposition}

\begin{proof} (a): Clearly $x_j^k\in(I(G)^k\colon x_i^k)$ for $x_j\in
N_G(x_i)$. Setting $H=G\setminus N_G(x_i)$, it is not hard to see that $x_j^k$ is a minimal generator
of the ideal $(I(G)^k\colon x_i^k)$ for $x_j\in N_G(x_i)$ and that any minimal
generator of $I(H)^k$ is a minimal generator of 
$(I(G)^k\colon x_i^k)$. The colon ideal 
$(I(G)^k\colon x_i^k)$ is minimally generated by 
$$
\{x_j^k\vert\, x_j\in
N_G(x_i)\}\cup G(I(H)^k)\cup\{x^{\alpha_1},\ldots,x^{\alpha_r}\},
$$
for some monomials $x^{\alpha_1},\ldots,x^{\alpha_r}$ such that each
$x^{\alpha_i}$ contains at least one variable in $N_G(x_i)$. One has
the equality 
$$
(N_G(x_i),(I(G)^k\colon x_i^k))=(N_G(x_i),I(H)^k).
$$
\quad Therefore, starting with the ideal $(I(G)^k\colon x_i^k)$ and
any variable $x_j$ in
$N_G(x_i)$, and successively applying
Theorem~\ref{morey-vila-oaxaca-2017}, the required inequality follows.

(c): Let $G_1,\ldots,G_m$ be the non-trivial connected components of
$G$. The analytic spread of $I(G_i)$ is equal to $|V(G_i)|$ if $G_i$
is non-bipartite and is equal to $|V(G_i)|-1$ otherwise 
(see \cite[Corollary~10.1.21 and
Proposition~14.2.12]{monalg-rev}). Hence the equality 
follows from the fact that the analytic spread is additive in
the sense of \cite[Lemma~3.4]{ass-powers}.

(e): This follows at once from part (d).
\end{proof}

\begin{lemma} Let $G$ be a bipartite graph with vertices
$x_1,\ldots,x_n$, let $I(G)$ be its edge
ideal, and let $k\geq 1$, $1\leq i\leq n$ be integers. 
The following hold.
\begin{itemize}
\item[(a)] $(I(G)\colon x_i)^k=(I(G)\colon x_i)^{(k)}$.
\item[(b)] $(I(G)^k\colon x_i^k)=(I(G)\colon
x_i)^k$. 
\end{itemize}
\end{lemma}

\demo (a): The graph $G\setminus N_G(x_i)$ is bipartite. 
Hence, according to \cite[Theorem 5.9]{ITG}, 
the ideal $I(G\setminus N_G(x_i))$ is normally torsion-free and so is
the ideal $(N_G(x_i))$ generated by $N_G(x_i)$. Therefore, by 
\cite[Corollary 5.6]{ITG}, the ideal $(I(G\setminus
N_G(x_i)),N_G(x_i))$ is normally torsion-free. Thus it suffices to
observe that $(I(G)\colon x_i)$ is equal to $(I(G\setminus
N_G(x_i)),N_G(x_i))$ (see \cite[p. 293]{monalg-rev}). 

(b): Let $\mathfrak{p}_1,\ldots,\mathfrak{p}_s$ be
the associated primes of $I(G)$. Since $G$ is bipartite, its edge ideal is normally
torsion-free \cite[Theorem 5.9]{ITG}. 
Therefore, using part (a) and noticing that the primary decomposition
of $(I(G)\colon x_i)$ is
$\cap_{x_i\notin\mathfrak{p}_j}\mathfrak{p}_j$, we get
\begin{eqnarray*}
(I(G)^k\colon
x_i^k)&=&\displaystyle\left(\left(\bigcap_{j=1}^s\mathfrak{p}_j\right)^k\colon
x_i^k\right)
=\left(\left(\bigcap_{j=1}^s\mathfrak{p}_j^k\right)\colon
x_i^k\right)=\bigcap_{x_i\notin\mathfrak{p}_j}\mathfrak{p}_j^k\\
&=&(I(G)\colon
x_i)^{(k)} =
(I(G)\colon x_i)^k.\ \Box
\end{eqnarray*}

The regularity of powers of the cover ideal of a bipartite graph was
studied in \cite{Jayanthan-etal} and the depth of symbolic powers of
cover ideals of graphs was examined in \cite{Hoa-etal,seyed-symbolic-covers}. 

\begin{corollary}\label{vila-hang} Let $G$ be a bipartite graph. The following hold. 
\begin{itemize}
\item[\rm(a)] \cite[Corollary~5.3]{kimura-terai-yassemi} If $G$ is unmixed, then $I(G)$ 
has non-increasing depth.
\item[\rm(b)] $($\cite[Theorem~3.2]{constantinescu-etal}, \cite{Nguyen-Thu-Hang},
\cite[Corollary~2.4]{Hang-Trung}$)$ $I(G)^\vee$ has non-increasing depth.
\item[\rm(c)] $I(G)^\vee$ has non-decreasing regularity. 
\end{itemize}
\end{corollary}

\begin{proof} (a): By \cite[Theorem~4.6 and Proposition~4.27]{reesclu}, the graph $G$ 
has the max-flow min-cut property and since $G$ is unmixed the result
follows at once from Theorem~\ref{winter2017}. 

(b)--(c): By \cite[Theorem~4.6 and Corollary~4.28]{reesclu}, 
$G^\vee$ has the max-flow min-cut property, and $G^\vee$ is unmixed
because its minimal vertex covers are the edges of $G$. 
Thus by Theorem~\ref{winter2017} the ideal $I(G^\vee)=I(G)^\vee$ has
non-increasing depth and non-decreasing regularity.
\end{proof}

The next interesting example is due to Kaiser,
Stehl\'\i k, and  \v{S}krekovski \cite{persistence-ce}. It shows that
the Alexander dual of a graph does not always has the
persistence property for associated primes. This example also shows
that part (b) of Corollary~\ref{vila-hang} fails for non-bipartite
graphs.

\begin{example}{\rm\cite{persistence-ce}}\label{kaiser}\rm\ 
Let $J=I^\vee$ be the Alexander dual of the edge ideal 
\begin{eqnarray*}
&
I=&(x_1x_2,x_2x_3,x_3x_4,x_4x_5,x_5x_6,x_6x_7,x_7x_8,x_8x_9,x_9x_{10},
x_1x_{10},x_2x_{11},x_8x_{11},\\
& &
x_3x_{12},x_7x_{12},x_1x_9,x_2x_8,x_3x_7,x_4x_6,x_1x_6,x_4x_9,x_5x_{10},
x_{10}x_{11},x_{11}x_{12},x_5x_{12}).
\end{eqnarray*}
Using {\it Macaulay\/}$2$ \cite{mac2}, it is seen that the values of $\depth(R/J^i)$,
for $i=1,\ldots,4$ are $8$, $5$, $0$, $4$, respectively. 
\end{example}

\begin{definition}\label{symbolic-non-inc-depth} Let $I\subset R$ be a
squarefree monomial ideal. The symbolic powers of $I$ have 
{\it non-increasing depth\/} if 
$$
\depth(R/I^{(k)})\geq\depth(R/I^{(k+1)})\ \forall\ k\geq 1,
$$
and have {\it non-decreasing regularity\/} if 
$\reg(R/I^{(k)})\leq\reg(R/I^{(k+1)})$ for all $k\geq 1$.
\end{definition}

If $G$ is a very well-covered graph (i.e., $G$ is unmixed, has no isolated vertices 
and $|V(G)|$ is equal to $2{\rm ht}(I(G))$), 
then the symbolic powers of $I(G)^\vee$ have 
non-increasing depth \cite{very-well-covered-non-inc} (cf.
\cite[Theorem~3.2]{Hoa-etal}) and the symbolic powers of
$I(G)$ have non-increasing depth
\cite[Theorem~5.2]{kimura-terai-yassemi}. 
The next result
complements these facts. 

\begin{proposition}\label{very-well-covered} 
If $G$ is a very well-covered graph, then the symbolic powers of
$I(G)$ have non-decreasing regularity. 
\end{proposition}

\begin{proof} The graph $G$ has a perfect matching by 
\cite[Corollary~3.7(ii)]{bounds-boletin}. Pick an edge $e$ in a
perfect matching of $G$ and set $x_e=\prod_{x_i\in e}x_i$. Note that 
any minimal vertex cover of $G$ intersects $e$ in exactly one vertex
because $G$ is unmixed. Therefore $(I^{(k+1)}\colon x_e)=I^{(k)}$ for
$k\geq 1$. Thus the result follows from part (iv) of 
Corollary~\ref{Caviglia-et-al-lemma-alternate-proof}.
\end{proof}

We will give another family of squarefree monomial ideals whose symbolic powers have
non-increasing depth. A {\it clique\/} of a graph $G$
is a set of vertices inducing a complete subgraph. The {\it clique
clutter\/} 
of $G$, denoted by ${\rm cl}(G)$, is the clutter on $V(G)$ whose edges are the 
maximal cliques of $G$ (maximal with respect to
inclusion).

\begin{definition}\rm A graph $G$ is called {\em strongly perfect}
 if every induced subgraph $H$ of $G$ has a maximal independent set
 of vertices $C$ such that 
$|C\cap e|=1$ for any maximal clique $e$ of $H$. 
\end{definition}

\begin{proposition}\label{strongly-perfect-symbolic} If $G$ is a strongly perfect graph and 
$J=I({\rm cl}(G)^\vee)$, then
\begin{itemize}
\item[(a)] $\depth(R/J^{(k)})\geq\depth(R/J^{(k+1)})\ \mbox{ for }\
k\geq 1$, and 
\item[(b)] $\reg(R/J^{(k)})\leq\depth(R/J^{(k+1)})\ \mbox{ for }\
k\geq 1$.
\end{itemize}
\end{proposition}

\begin{proof} Let $\mathfrak{p}_1,\ldots,\mathfrak{p}_s$ be the set of
all ideals $(e)$ such that $e\in E({\rm cl}(G))$.  From the equality 
$$
J=I({\rm cl}(G)^\vee)=I({\rm cl}(G))^\vee=
\bigcap_{e\in E({\rm cl}(G))}(e)=\bigcap_{i=1}^s\mathfrak{p_i},
$$
we get $J^{(k)}=\cap_{i=1}^s\mathfrak{p}_i^k$ for $k\geq 1$. As $G$ is strongly
perfect, $G$ has a maximal independent set of vertices $C$ such that
$|C\cap e|=1$ for any $e\in{\rm cl}(G)$, that is, 
$|C\cap \mathfrak{p}_i|=1$ for $i=1,\ldots,s$. Hence, setting
$f=\prod_{x_i\in C}x_i$, one has the equalities
$$
(J^{(k+1)}\colon f)=\left(\bigcap_{i=1}^{s}\mathfrak{p_i}^{k+1}\colon
f\right)=\bigcap_{i=1}^s(\mathfrak{p}_i^{k+1}\colon f)= 
\bigcap_{i=1}^s\mathfrak{p}_i^{k}=J^{(k)} \ \mbox{ for }\ k\geq 1. 
$$
\quad Therefore, by parts (ii) and (iv) of
Corollary~\ref{Caviglia-et-al-lemma-alternate-proof}, 
one has 
$$
\depth(R/J^{(k)})=\depth(R/(J^{(k+1)}\colon
f))\geq\depth(R/J^{(k+1)}), 
$$
and $\reg(R/J^{(k)})=\reg(R/(J^{(k+1)}\colon
f))\leq\reg(R/J^{(k+1)})$.
\end{proof}
\begin{proposition}\label{additivity-cm-square}
Let $A=K[X]$ and $B=K[Y]$ be polynomial rings over a field $K$ in disjoint sets
of variables, let $I$ and $J$ be nonzero homogeneous proper ideals of $A$ and $B$
respectively, and let $R=K[X,Y]$. The following hold. 
\begin{itemize}
\item[(a)] \cite[Proposition
3.7]{Ha-sum-powers-of-sums} $R/(I + J)^i$ is 
Cohen--Macaulay for all $i\leq k$ if and only if $A/I^i$ and
$B/J^i$ are Cohen--Macaulay for all $i\leq k$.
\item[(b)] \cite[Lemma~3.2]{Hoa-mixed-product}
$\reg(R/(I+J))=\reg(A/I)+\reg(B/J)$.
\item[(c)] \cite[Lemma~3.2]{Hoa-mixed-product}
$\reg(R/IJ)=\reg(A/I)+\reg(B/J)+1$.
\end{itemize}
\end{proposition}

The Cohen--Macaulay property of the square of an edge ideal can be
expressed in terms of its connected components (cf.
\cite[Lemma~4.1]{Vi2}). For additional results
on the depth of powers of sums of ideals see
\cite{Ha-sum-powers-of-sums} and the references therein.

\begin{corollary}{\cite[Corollary~4.9]{rinaldo-terai-yoshida}}\label{cm-square-components} 
Let $G$ be a graph with connected components
$G_1,\ldots,G_m$. Then $I(G)^2$ is Cohen--Macaulay if and only if
$I(G_i)^2$ is Cohen--Macaulay for $i=1,\ldots,m$.
\end{corollary}

\begin{proof} Since the radical of a Cohen--Macaulay monomial 
ideal is Cohen--Macaulay \cite{herzog-takayama-terai} (see
Corollary~\ref{herzog-takayama-terai-theo}), the results follows from
Proposition~\ref{additivity-cm-square}.
\end{proof}

\begin{example}\label{depth-is-not-additive}
Let $A=K[x_1,x_2,x_3]$ and $B=K[y_1,y_2,y_3]$ be polynomial rings over
a field $K$, let $I=(x_1x_2,x_2x_3,x_1x_3)$ and
$J=(y_1y_2,y_2y_3,y_1y_3)$ be ideals of $A$ and $B$
respectively, and let $R=K[X,Y]$. Then $A/I^2$ and $B/I^2$ have depth $0$
but $R/(I+J)^2$ has depth $1$, that is, the depth of squares of
monomial ideals is not additive on disjoint sets of variables. 
\end{example}

\begin{lemma}\label{square-cm-lemma} 
Let $G$ be a graph without isolated vertices. The
following hold.
\begin{itemize}
\item[(a)] If $R/I(G)^2$ is Cohen--Macaulay, then  $R/(I(G\setminus
N_G(x_i))^2$ is Cohen--Macaulay for any $x_i$.
\item[(b)] $\depth(R/I(G)^2)=0$ if and only if $G$ has a triangle
$C_3$ that intersects $N_G(x_i)$ for any $x_i$ outside $C_3$. In
particular if the depth of $R/I(G)^2$ is $0$, then $G$ is connected. 
\end{itemize}
\end{lemma} 

\begin{proof} (a): Using Proposition~\ref{trung-limit-powers}(a) and
Corollary~\ref{Caviglia-et-al-lemma-alternate-proof}(ii), we get
$$\depth(R/(I(G\setminus
N_G(x_i))^2,N_G(x_i)))\geq\depth(R/(I(G)^2\colon
x_i^2))\geq\depth(R/I(G)^2)$$
for all $i$. Thus $R/(I(G\setminus N_G(x_i))^2$ is Cohen--Macaulay for
all $i$. 

(b) ($\Rightarrow$): As $\mathfrak{m}=(x_1,\ldots,x_n)$ is an
associated prime of $I(G)^2$, there is $x^a=x_1^{a_1}\cdots x_n^{a_n}$
such that $(I(G)^2\colon x^a)=\mathfrak{m}$. Thus $x_ix^a\in I(G)^2$
for all $i$ and $x^a\notin I(G)^2$. Note that $x^a$ is squarefree.
Indeed if $a_k\geq 2$ for some $k$, then $x_kx^a=x^bf_if_j$ for some
monomial $x^b$ and some minimal generators $f_i,f_j$ of $I(G)$, which
is impossible because $f_i,f_j$ are squarefree monomials of 
degree $2$ and $x^a\notin I(G)^2$. Thus we may assume that
$x^a=x_1\cdots x_r$, for some $r\geq 3$, and $x_1x_2\in I(G)$. Then 
$x_3x^a=x^bf_if_j$ for some $x^b$ and some minimal generators
$f_i,f_j$ of $I(G)$. One can write $f_i=x_3x_k$ and
$f_j=x_3x_\ell$, $k\neq \ell$, $k\neq 3$, $\ell\neq 3$. Clearly either
$x_k=x_1$ or $x_k=x_2$ and either $x_\ell=x_1$ or $x_\ell=x_2$ because
$x^a$ is not in $I(G)^2$. Thus $x_1,x_2,x_3$ are the vertices of a
triangle of $G$ that we denote by $C_3$. Since $x_rx^a\in I(G)^2$, 
it follows that $r=3$. Take any vertex $x_k$ not in $C_3$. As
$x_kx^a=x_k(x_1x_2x_3)$ and $x_kx^a$ is in $I(G)^2$, we get that $x_k$
is adjacent to some vertex of $C_3$.

(b) ($\Leftarrow$): Pick a triangle $C_3$ of $G$ such that any vertex
outside $C_3$ is adjacent to a vertex of $C_3$. Setting 
$x^a=\prod_{x_i\in V(C_3)}x_i$, we get that $(I(G) ^2\colon x^a)$ is
the maximal ideal $\mathfrak{m}=(x_1,\ldots,x_n)$. Thus $\mathfrak{m}$
 is an associated prime of $I(G)^2$, that is, $\depth(R/I(G)^2)=0$.
 This part could also follow from a general construction of \cite{AJ}.
\end{proof}

In \cite{Hoang-etal,hoang-gorenstein-second-jaco} the Cohen--Macaulay
property of the square of the edge ideal of a graph is classified. 

\begin{theorem}{\cite[Theorem~4.4]{hoang-gorenstein-second-jaco}}
\label{cm-second-power} 
Let $G$ be a graph with vertex set $V(G)=\{x_1,\ldots,x_n\}$
and without isolated vertices. Then $I(G)^2$ 
is Cohen--Macaulay if and only if $G$ is a triangle-free unmixed graph
and $G\setminus\{x_i\}$ is unmixed for all $i$. 
\end{theorem}

As an application we recover the following facts. 

\begin{corollary}{\rm(\cite[Theorem~2.7]{crupi-et-al}, 
\cite[Proposition~4.2]{Hoang-etal})}\label{square-cm} 
Let $G$ be a bipartite graph without isolated vertices. Then 
$I(G)^2$ is Cohen--Macaulay if and only if $I(G)$ is a complete
intersection, i.e., $G$ is a disjoint union of edges. 
\end{corollary} 

\begin{proof} $\Rightarrow$): Since $I(G)$ is the radical of $I(G)^2$,
by Corollary~\ref{herzog-takayama-terai-theo}, the ideal $I(G)$ is 
Cohen--Macaulay. Hence, according to a structure theorem 
for Cohen--Macaulay bipartite graphs \cite[Theorem~3.4]{herzog-hibi-crit},
there is a bipartition $V_1=\{x_1,\ldots,x_g\}$,
$V_2=\{y_1,\ldots,y_g\}$ of $G$ such that: 

\noindent {\rm (i)} $\{x_i,y_i\}\in E(G)$ for all $i$, 

\noindent{\rm (ii)} if $\{x_i,y_j\}\in E(G)$, then $i\leq j$, and 

\noindent {\rm (iii)} if $\{x_i,y_j\}$, $\{x_j,y_k\}$ are in $E(G)$ and $i<j<k$, 
then $\{x_i,y_k\}\in E(G)$.

We proceed by induction on $g$. If $g=1$, $I(G)$ is clearly a complete
intersection. Using the connected components of $G$ together with 
Corollaries~\ref{herzog-takayama-terai-theo} and \ref{cm-square-components}, 
and Proposition~\ref{additivity-cm-square}, we may assume that $I(G)^2$ is
Cohen--Macaulay and that $G$ is a Cohen--Macaulay connected bipartite
graph. Consider the graph $H=G\setminus N_G(y_1)$. We set 
$R=K[V_1\cup V_2]$. Note that $N_G(y_1)=\{x_1\}$. Hence, by
Lemma~\ref{square-cm-lemma}(a), $I(G\setminus\{x_1\})^2$ is
Cohen--Macaulay and so is $I(G\setminus\{x_1\})$. Therefore by
induction $I(G\setminus\{x_1\})$ is generated by 
$x_2y_2,\ldots,x_gy_g$. As $G$ is connected, using
(i)--(iii), it is seen that the edges of $G$ are the edges of the
perfect matching and all edges of the form $\{x_1,y_i\}$, $i\geq 1$.
It is not hard to see (by a separate induction procedure) that the 
square of $I(G)$ is not Cohen--Macaulay if $g\geq 2$. Thus $g=1$. 

$\Leftarrow)$: If $I(G)$ is a complete intersection, it is well 
known that all powers of $I(G)$ are Cohen--Macaulay \cite[17.4, p.~139]{Mats}.
\end{proof}

Let $G$ be a graph. The next corollary follows from the result that
``$I(G)^2 = I(G)^{(2)}$ if
and only if $G$ has no triangles''. This result originated in
\cite{ITG} implicitly, written explicitly
in \cite[Lemma~3.1]{rinaldo-terai-yoshida-1}. 
Fix an integer $t\geq 2$. This lemma shows that 
 $I(G)^t = I(G)^{(t)}$ if and only if $G$ contains no odd cycles of length $2s-1$ 
for any $2\leq s\leq t$ (cf. \cite[Theorem~4.13]{Dao-stefani-grifo-huneke-betancourt}). 

\begin{corollary}\label{square-cm-1} 
Let $G$ be a graph without isolated vertices. If $I(G)^2$ is
Cohen--Macaulay, then $G$ has no triangles. 
\end{corollary} 

\begin{proof} 
Let $V(G)=\{x_1,\ldots,x_n\}$ be the vertex set of $G$ and let
$R$ be the polynomial ring $K[V(G)]$. We proceed by induction on $n$.
The result is clear for $n=1,2,3$. Assume $n\geq 4$. We proceed by
contradiction assuming that $G$ has a triangle $C_3$. 
Using the connected components of $G$ together with 
Corollaries~\ref{herzog-takayama-terai-theo} and \ref{cm-square-components}, 
and Proposition~\ref{additivity-cm-square}, we may assume that $I(G)^2$ is
Cohen--Macaulay and $G$ is connected. Thus, by
Lemma~\ref{square-cm-lemma}(a), $I(G\setminus N_G(x_i))^2$ is Cohen--Macaulay for
all $i$. If $G$ has a vertex $x_i$ not in $C_3$
such that $N_G(x_i)$ do not intersect the vertex set $V(C_3)$ of $C_3$, then
$C_3$ is a triangle of $G\setminus N_G(x_i)$, a contradiction. Thus 
any vertex outside $C_3$ is adjacent to a vertex of $C_3$. Hence, by
Lemma~\ref{square-cm-lemma}(b), we get $\depth(R/I(G)^2)=0$, a
contradiction.  This part could also follow from a general
construction of \cite{AJ}. 
\end{proof}

\begin{example}{\cite{hoang-gorenstein-second-jaco,rinaldo-terai-yoshida}}\label{Hoang-example}
The square of the edge ideal of the graph $G$ of
Fig.~\ref{graph1} is Cohen--Macaulay and $I(G)$ is Gorenstein. This
can be verified using {\em Macaulay\/}$2$ \cite{mac2}. This example 
appears as a special case of \cite[Conjecture~5.7]{rinaldo-terai-yoshida}. A result of Hoang
and Trung \cite[Theorem~4.4]{hoang-gorenstein-second-jaco} 
shows that for a graph $G$ without isolated vertices $I(G)^2$ is
Cohen--Macaulay if and only if $G$ is triangle-free and 
Gorenstein. The Cohen-Macaulay property of $I(G)^2$ 
is also studied in \cite{trung-tuan} in terms
of simplicial complexes.   
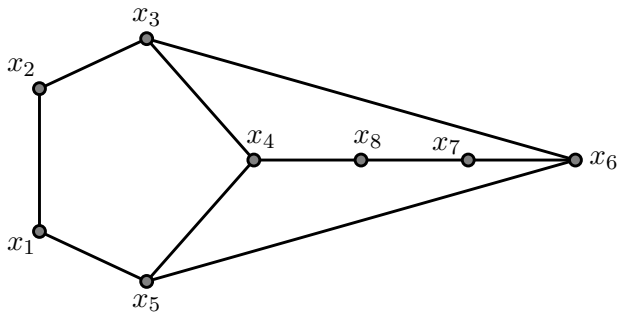
\begin{figure}[h]
\begin{center}
	\begin{tikzpicture}[line width=1.1pt,scale=0.95]
	\tikzstyle{every node}=[inner sep=0pt, minimum width=4.5pt]
	\draw [-] (-1.5,-1)--(-1.5,1); 
	\draw [-] (-1.5,1)--(0,1.7); 
	\draw [-] (-1.5,-1) --(0,-1.7); 
	\draw [-] (0,-1.7)--(1.5, 0); 
	\draw [-] (0,1.7)--(1.5,0); 
	\draw [-] (1.5,0)--(3,0); 
	\draw [-] (3,0)--(4.5,0); 
	\draw [-] (4.5,0)--(6,0); 
	\draw [-] (0,-1.7)--(6,0); 
	\draw [-] (0, 1.7)--(6,0); 
	\draw (-1.5,-1.0) node (v1) [draw, circle, fill=gray] {};
	\draw (0,-1.7) node (v5) [draw, circle, fill=gray] {};
	\draw (1.5,0) node (v4) [draw, circle, fill=gray] {};
	\draw (0,1.7) node (v3)[draw, circle, fill=gray] {};
	\draw (-1.5,1.0) node (v2) [draw, circle, fill=gray] {};
	\draw (6, 0) node (v6) [draw, circle, fill=gray] {};
	\draw (4.5, 0) node (v7) [draw, circle, fill=gray] {};
	\draw (3, 0) node (v8) [draw, circle, fill=gray] {};
	\node at (-1.75,-1.2) {$x_{1}$};
	\node at (0,-2) {$x_{5}$};
	\node at (1.6,.29) {$x_{4}$};
	\node at (0,2) {$x_{3}$};
	\node at (-1.75,1.3) {$x_{2}$};
	\node at (3.1,.29) {$x_{8}$};
    \node at (4.2,.20) {$x_{7}$};
    \node at (6.4,0) {$x_{6}$};
	\end{tikzpicture}
\caption{Gorenstein Graph $G$.}\label{graph1}
\end{center}
\end{figure}
\end{example}


\bibliographystyle{plain}

\end{document}